\documentclass[11pt,reqno,tbtags]{amsart}
\usepackage{amssymb}
\usepackage{natbib}
\bibpunct[, ]{[}{]}{;}{n}{,}{,}
\numberwithin{equation}{section}
\allowdisplaybreaks

\newtheorem{theorem}{Theorem}[section]
\newtheorem{lemma}[theorem]{Lemma}

\newtheorem{corollary}[theorem]{Corollary}
\newtheorem{question}[theorem]{Question}

\theoremstyle{definition}
\newtheorem{example}[theorem]{Example}

\newtheorem{remark}[theorem]{Remark}

\newtheorem*{acks}{Acknowledgements}

\theoremstyle{remark}

\newenvironment{romenumerate}{\begin{enumerate}
 }{\end{enumerate}}

\newcounter{oldenumi}
{\setcounter{oldenumi}{\value{enumi}}
\begin{romenumerate} \setcounter{enumi}{\value{oldenumi}}}
{\end{romenumerate}}

\newcounter{thmenumerate}
\newenvironment{thmenumerate}
{\setcounter{thmenumerate}{0}%
 \def\item{\par
 \refstepcounter{thmenumerate}\textup{(\roman{thmenumerate})\enspace}}
}
{}

\newcounter{xenumerate}   

\newcommand{\refT}[1]{Theorem~\ref{#1}}
\newcommand{\refC}[1]{Corollary~\ref{#1}}
\newcommand{\refL}[1]{Lemma~\ref{#1}}
\newcommand{\refR}[1]{Remark~\ref{#1}}
\newcommand{\refS}[1]{Section~\ref{#1}}

\newcommand{\refE}[1]{Example~\ref{#1}}

\newcommand{\refQ}[1]{Question~\ref{#1}}

\newcommand{\refand}[2]{\ref{#1} and~\ref{#2}}

\newcommand\nopf{\qed}   

\newcommand{\sumik}{\sum_{i=1}^k}
\newcommand{\sumim}{\sum_{i=1}^m}
\newcommand{\sumir}{\sum_{i=1}^r}
\newcommand{\sumjm}{\sum_{j=1}^m}
\newcommand{\prodim}{\prod_{i=1}^m}
\newcommand{\prodir}{\prod_{i=1}^r}

\newcommand\set[1]{\ensuremath{\{#1\}}}
\newcommand\bigset[1]{\ensuremath{\bigl\{#1\bigr\}}}
\newcommand\Bigset[1]{\ensuremath{\Bigl\{#1\Bigr\}}}

\newcommand\xpar[1]{(#1)}
\newcommand\bigpar[1]{\bigl(#1\bigr)}
\newcommand\Bigpar[1]{\Bigl(#1\Bigr)}
\newcommand\biggpar[1]{\biggl(#1\biggr)}

\newcommand\bigabs[1]{\bigl|#1\bigr|}

\newcommand\biggabs[1]{\biggl|#1\biggr|}
\newcommand\lrabs[1]{\left|#1\right|}
\def\rompar(#1){\textup(#1\textup)}    
\newcommand\xfrac[2]{#1/#2}

\newcommand\parfrac[2]{\Bigpar{\frac{#1}{#2}}}

\newcommand\ceil[1]{\lceil#1\rceil}
\newcommand\floor[1]{\lfloor#1\rfloor}

\newcommand\ntoo{\ensuremath{{n\to\infty}}}

\newcommand\ktoo{\ensuremath{{k\to\infty}}}

\newcommand\xtoo{\ensuremath{{x\to\infty}}}

\newcommand\ie{i.e.\spacefactor=1000}

\newcommand\etc{etc.\spacefactor=1000}

\newcommand\ii{\mathrm{i}}

\newcommand{\tend}{\longrightarrow}

\newcommand\pto{\overset{\mathrm{p}}{\tend}}

\newcommand\eqd{\overset{\mathrm{d}}{=}}

\newcommand\bbR{\mathbb R}
\newcommand\bbC{\mathbb C}

\newcommand\bbZ{\mathbb Z}
\newcommand\bbQ{\mathbb Q}

\newcounter{CC} 
\newcounter{cc}

\newcommand\E{\operatorname{\mathbb E{}}}
\renewcommand\P{\operatorname{\mathbb P{}}}
\newcommand\Ps{\operatorname{\mathcal P{}}}

\newcommand\diam{\operatorname{diam}}

\newcommand\ga{\alpha}
\newcommand\gb{\beta}
\newcommand\gd{\delta}

\newcommand\gam{\gamma}

\newcommand\gl{\lambda}

\newcommand\go{\omega}

\newcommand\eps{\varepsilon}

\newcommand\cF{\mathcal F}
\newcommand\cL{{\mathcal L}}
\newcommand\cP{\mathcal P}
\newcommand\cS{{\mathcal S}}
\newcommand\cT{{\mathcal T}}

\newcommand\tS{{\tilde S}}
\newcommand\hT{{\widehat T}}

\newcommand\limn{\lim_{n\to\infty}}

\newcommand\qw{^{-1}}
\newcommand\qww{^{-2}}

\renewcommand{\=}{:=}

\newcommand\dd{\,\textup{d}}

\newcommand\nn{\tau}
\newcommand\kk{\kappa}
\newcommand\eff{\varepsilon}
\newcommand\bbzn{\bbZ_n}

\newcommand\nngr{\nn_{\textup{Greedy}}}
\newcommand\mm[1]{\par \emph{Method #1}.}
\newcommand\mmx[1]{Method~#1}
\newcommand\q{\bar}
\newcommand\qa{\overline}
\newcommand\qx{\q x}
\newcommand\ff{\widehat}
\newcommand\xix{\wXi^*}
\newcommand\bbznx{\bbzn^*}
\newcommand\cznx{{\chi^*_n}}
\newcommand\ffznx{\widehat{\cznx}}
\newcommand\xxx{^*}
\newcommand\xHaa{M(a_1,\dots,a_m)}
\newcommand\pie{p_i^{e_i}}
\newcommand\dela{\,|\,}
\newcommand\ndela{\nmid}
\newcommand\ex[1]{\exp(#1)}

\newcommand\urladdrx[1]{{\urladdr{\def~{{\tiny$\sim$}}#1}}}

\renewcommand\Pr{{\mathbb P}}
\newcommand{\GS}{G_S}
\newcommand\worst{\sigma}

\newcommand\cd{covering multiplicity}
\newcommand\wXi{W}
\newcommand\weps{\gd}

\begin{document}
\title
{On covering by translates of a set}

\date{October 19, 2009; revised May 25, 2010}

\author{B\'ela Bollob\'as}
\address{Department of Pure Mathematics and Mathematical Statistics,
Wilberforce Road, Cambridge CB3 0WB, UK and
Department of Mathematical Sciences,
University of Memphis, Memphis TN 38152, USA}
\thanks{The first author's research was supported in part by NSF grants CNS-0721983, CCF-0728928
and DMS-0906634, and ARO grant W911NF-06-1-0076.}
\email{b.bollobas@dpmms.cam.ac.uk}

\author{Svante Janson}
\address{Department of Mathematics, Uppsala University, PO Box 480,
SE-751~06 Uppsala, Sweden}
\email{svante.janson@math.uu.se}
\urladdrx{http://www.math.uu.se/~svante/}

\author{Oliver Riordan}
\address{Mathematical Institute, University of Oxford, 24--29 St Giles', Oxford OX1 3LB, UK}
\email{riordan@maths.ox.ac.uk}


\begin{abstract} 
In this paper we study the minimal number $\nn(S,G)$ of translates
of an arbitrary subset $S$ of a group $G$ needed to cover the group,
and related notions of the efficiency of such coverings.
We focus mainly on finite subsets in discrete groups, reviewing the
classical results in this area, and generalizing them to
a much broader context. For example,
the worst-case efficiency when $S$ has $k$ elements
is of order $1/\log k$. We show that if $n(k)$ grows
at a suitable rate with $k$, then almost every $k$-subset
of any given group with order $n$ comes close to this worst-case bound.
In contrast, if $n(k)$ grows very rapidly, or if $k$ is fixed and
$n\to\infty$,
then almost every $k$-subset of the cyclic group with order $n$
comes close to the optimal efficiency.
\end{abstract}

\maketitle

\section{Introduction}\label{S:intro}

Packing and covering problems have been studied for many decades, both
in discrete geometry and in stochastic geometry.
One of the basic questions of discrete geometry is the following:
given sets $S$, $T\subset\bbR^d$, what
is the minimal number of translates of $S$ that cover $T$?
The most studied cases are when $S$ is a ball in $\bbR^d$, or 
a convex polygon in $\bbR^2$. For these and other questions see
the classical treatises of Fejes T\'oth~\cite{FT53,FT72} and Rogers~\cite{Rogers}.
For a selection of classical and more recent results see, for example,
\cite{BambahRogers,BRZ,Boroczky,ERogers,gFT97,Hall1,PachAgarwal}.

In stochastic geometry similar questions are considered. Here the starting point
was the study of the probability that $n$ random arcs, each of length
$a$, cover the entire circle; see Steutel~\cite{Steutel},
Flatto~\cite{Flatto} and Janson~\cite{Janson1,Janson3}.
For general random covering problems see the books by Penrose~\cite{Penrose03},
Meester and Roy~\cite{MeesterRoy} and Hall~\cite{Hall}, and a host of papers including
\cite{AppelRusso,ARS,BBSW,Janson2,Maehara,Penrose99}.

In almost all cases above, the set $S$ is taken to be convex. In this paper we shall address
a rather different kind of covering problem, studying the minimal number of translates of
a given, arbitrary subset $S$ of a general group $G$ needed to cover the group.
Of course, in this generality not too much can be said, so we shall focus on various
natural special cases. Although the general problem includes those concerning balls
mentioned above, the flavour of the cases we study turns out to be closer
to covering problems in stochastic geometry.
Some of the special cases we consider have been studied before:
for example, Newman~\cite{Newman_dens}
studied the density of coverings of $\bbZ$ by a given finite set
(in fact, he studied coverings of the natural numbers, which turns
out to be equivalent),
and Schmidt~\cite{Schmidt} proved results about coverings of $\bbZ^d$ (as well as various other cases).
We shall comment on the relationship of our results
to these earlier results in the relevant sections.

Let $G$ be a group and let $S$ be a  non-empty subset of $G$.
By the \emph{covering number} $\nn(S,G)$ we mean the smallest number of
(left) translates of $S$ that cover $G$,
\ie,
\begin{equation}\label{tsg}
  \begin{split}
  \nn(S,G)
&\=
\min\bigset{m:\bigcup_{i=1}^m t_i S=G 
  \text{ for some $t_1,\dots,t_m\in G$}}
\\&\phantom:
=\min\set{|T|:TS=G}.	
  \end{split}
\end{equation}
If no finite set of translates of $S$ covers $G$, then we set
$\nn(S,G)=\infty$. Our aim is to study the efficiency
of coverings by $S$, loosely defined as the reciprocal
of the average number of times an element is covered
in a covering achieving the minimum $\nn(S,G)$.

We are mainly interested in Abelian groups, for which we
use additive notation, and in particular the
groups $\bbZ_n$, $\bbZ$ and $\bbR$, and their powers. In general, however, $G$
can be any group, and need not be Abelian.

In the following sections we consider some natural special cases, describing
the appropriate notions of \cd\ and covering efficiency, and proving various
results. Our main results focus on the worst possible efficiency of sets
$S$ in some class, for example the class of $k$-element subsets of $\bbZ$. We also
include many examples and trivial results to illustrate the definitions,
and, perhaps most importantly, we pose many open problems.

After giving the definitions for general compact groups in \refS{Scompact},
we first consider finite groups, presenting our basic results in \refS{Sfinite},
and considering random subsets $S$ in \refS{Srandom}.
We turn to subsets of $\bbZ$ in \refS{SZ}. In Sections~\ref{Slc} and~\ref{SZsmall},
we compare coverings of $\bbZ$ with those of cyclic groups,
focusing on small sets $S$. In Sections~\ref{SR}--\ref{SR2} we consider
the case $G=\bbR$. Finally, in Sections~\ref{SZd} and~\ref{SRd} we consider
coverings of $\bbZ^d$ and $\bbR^d$, respectively; these sections
contain (essentially) only questions, rather than results.

\begin{acks}
  This research was begun during a visit by SJ to the University of
  Cambridge, partly funded by Trinity College, Cambridge, and continued during a
  visit by SJ to the Isaac Newton Institute in Cambridge, funded by a
  Microsoft fellowship. We are grateful to two anonymous referees for careful
  reading of the paper, for finding two minor errors, and for helpful suggestions concerning the presentation.
\end{acks}

\section{$G$ compact}\label{Scompact}

If $G$ is a compact group, 
it has a finite Haar
measure, which we denote by $\mu$.
We consider only the case when 
$S$ is measurable and 
$\nn(S,G)<\infty$; note that the latter holds whenever
$S$ has an interior point.

When $G$ is compact and $\nn(S,G)<\infty$,
we define the \emph{\cd} of $S$ as
\begin{equation}
  \label{ksg}
\kk(S,G)\=\frac{\nn(S,G)\,{\mu(S)}}{\mu(G)},
\end{equation}
this is 
the average number of times each point in $G$ is covered by a
smallest (minimum cardinality) covering by translates of $S$.
The \emph{efficiency} of $S$ is
\begin{equation}\label{esg}
  \eff(S,G)\=1/\kk(S,G).
\end{equation}
Note that these definitions do not depend on the chosen normalisation of the Haar
measure $\mu$. 

If $TS=G$, then 
$\mu(G)=\mu\bigpar{\bigcup_{t\in T} tS}\le |T|\mu(S)$, so 
trivially
\begin{align}
&\phantom{<}\nn(S,G)\ge\frac{\mu(G)}{\mu(S)}
\label{a0}
\intertext{and thus}
    1&\le\kk(S,G)<\infty,
\label{a1}\\
0&<\eff(S,G)\le1.
\label{a2}
\end{align}

We also define the \emph{\cd} $\kk(\cT)$ or \emph{efficiency} $\eff(\cT)$  
of a particular covering $\cT=\set{t_iS}$ of $G$:
\begin{equation*}
 \kk(\cT)\=|\cT|\mu(S)/\mu(G) \qquad\hbox{and}\qquad
\eff(\cT)\=1/\kk(\cT).
\end{equation*}
The analogues of \eqref{a1} and \eqref{a2} hold for these too.
Trivially, $\kk(S,G)=\min\{\kk(\cT)\}$, where the minimum
is over coverings $\cT$ by translates of $S$, and $\eff(S,G)=\max\{\eff(\cT)\}$.

We call a subset $S$ of $G$ \emph{efficient}, or \emph{efficiently covering},
if $\kk(S,G)=\eff(S,G)=1$.
Typically, $\kk(S,G)>1$ and thus $\eff(S,G)<1$, in which
case $S$ is \emph{inefficient}.
Of course, there
many efficient sets $S$; here is one simple family of examples.

  \begin{example}\label{Esubgroup}
If $S$ is a subgroup of $G$ of finite index, then $\nn(S,G)=\mu(G)/\mu(S)$ and 
$\kk(S,G)=\eff(S,G)=1$.
  \end{example}

In general, the study of efficient coverings, or, essentially equivalently, 
of \emph{tilings}, has received much more attention than the inefficient
case. One early example is the paper of Haj\'os~\cite{Hajos}; we shall mention
some other examples in specific contexts later.
Here we focus on the inefficient case, and specifically on
the question of how large $\kk(S,G)$ is, rather than simply whether it is equal to $1$
or not.

\begin{remark}
\label{Reffsize}
We may also define the \emph{effective size} of $S$ as
\begin{equation*}
\nu(S,G)\=\frac{\mu(G)}{\nn(S,G)}=\frac{\mu(S)}{\kk(S,G)}
=\eff(S,G)\mu(S),
\end{equation*}
which satisfies
$0<\nu(S,G)\le\mu(S)$. Thus, as far as coverings are concerned,
$S$ is `as good as' an efficiently covering set of measure $\nu(S,G)$.
Analogous definitions with $\nu(S)\=\eff(S)\mu(S)$ can be made for
the other cases studied in later sections,
in particular for $G=\bbZ$ and $G=\bbR$.
Although the quantity $\nu$ also carries intuitive information, we shall
work throughout with $\nn$, $\kk$ and $\eff$.
\end{remark}

\section{$G$ finite}\label{Sfinite}
An important special case is when $G$ is
finite, in which case we use the counting
measure $\mu(A)=|A|$ as the (normalized) Haar measure.
Note that now $\nn(S,G)\le|G|<\infty$ for every non-empty set $S$.
Hence, for any non-empty $S\subseteq G$, 
\begin{align}
\frac{|G|}{|S|}&\le \nn(S,G) \le |G|,
\label{b0}\\
    1&\le\kk(S,G)\le|S|,
\label{b1}\\
|S|\qw&\le\eff(S,G)\le1.
\label{b2}
\end{align}

The lower bounds on $\nn$ and $\kk$, or, equivalently, the upper bound
$\eff\le1$, can be attained,
for example in Examples \refand{Esubgroup}{Eresidue}.
Indeed, when $G$ is finite, a covering $\cT=\set{t_i S}$ is
\emph{efficient} (\ie, $\eff(\cT)=1$) if and only if it is a partition
of $G$ into disjoint translates of $S$.
On the other hand, the lower bound on $\eff$  
(upper bounds on $\nn$ and $\kk$)
can be substantially improved as we shall now see.

One method of finding a covering set of translates is the greedy
algorithm. Pick elements
$t_1,t_2,\dots$ in $G$ one by one as follows:
having chosen $t_1,\dots,t_{j-1}$
with $\bigcup_{i=1}^{j-1} t_iS\neq G$, choose $t_j$ so as to
maximize
$\bigabs{\bigcup_{i=1}^j t_iS}$. (If several choices of $t_j$
achieve the maximum, pick one of them by any rule.)
The algorithm terminates when
$\bigcup_{i=1}^{j} t_iS= G$. We shall write $\nngr(S,G)$
for the final value of $j$, i.e., the number of translates
of $S$ in the covering produced by the greedy algorithm.

\begin{theorem}
  \label{Tgreedy1}
Suppose that $G$ is a finite group with $|G|=n$, and let
$S\subseteq G$ with $|S|=k\ge1$.
Define integers $n_j$
recursively by $n_0=n$ and
\begin{equation}\label{njdef}
  n_{j+1}=\floor{n_j(1-k/n)},
\qquad j\ge1.
\end{equation}
Then
\begin{equation}\label{g2a}
\nn(S,G)\le
  \nngr(S,G)
\le\min\set{j:n_j=0}.
\end{equation}
\end{theorem}

\begin{proof}
Let $m_j\=
\bigabs{{G\setminus\bigcup_{i=1}^j t_iS}}$ be the number of uncovered
elements after $j$ steps of the greedy algorithm.
Since a random choice of $t_{j+1}$ covers on average $m_j|S|/|G|$
of these, the optimal choice covers at least as many and thus
\begin{equation}\nonumber
  m_{j+1}\le m_j - m_j|S|/|G| = m_j(1-k/n).
\end{equation}
Since $m_{j+1}$ is an integer we have $m_{j+1}\le \floor{m_j(1-k/n)}$,
and so, by induction, $m_j\le n_j$ for every $j\ge0$. In
particular, if $n_j=0$ then $m_j=0$ and thus $\bigcup_{i=1}^jt_i S=G$,
which yields \eqref{g2a}.
\end{proof}

The simplest way to obtain an explicit bound on $\nn(S,G)$ from \refT{Tgreedy1}
is to observe that \eqref{njdef} implies $n_j\le n(1-k/n)^j<ne^{-jk/n}$.
Setting $\ell=\frac{n}{k}\log n$, this gives $n_{\ceil{\ell}}<1$,
so $n_{\ceil{\ell}}=0$ and $\nn(S,G)\le \ceil{\ell}$. In fact, we can do
better.

Let
 $H_k\=\sum_{j=1}^k 1/j$
denote the $k$th harmonic number. It is well known
that
$\log k < H_k \le \log k+1$.

\begin{corollary}\label{Cgreedy}
Let $G$, $S$, $n$, $k$ be as in \refT{Tgreedy1}.
Then
\begin{align}
  \nn(S,G)&\le\frac{n}k H_k 
\le\frac nk (\log k+1), \label{g2f}
\\
  \kk(S,G)&\le H_k\le \log k+1,\label{g2fk}
\\
  \eff(S,G)&\ge \frac1{H_k}\ge \frac1{\log k+1}. \label{g2fe}
\end{align}
\end{corollary}

\begin{proof}
Let $N(t)$ be defined for $t\ge0$ 
by $N(j)=n_j$ for integer $j$ and by linear interpolation between integers.
Then $N(t)$ is continuous and non-increasing, with $N(0)=n$ and
$N(t)=0$ for large $t$.
For $0\le i\le k$, let $t_i\=\min\set{t:kN(t)/n\le i}$; thus 
$0=t_k<t_{k-1}<\dots< t_0$.

If $t>0$ is not an integer, then, setting $j\=\floor t$, so $j<t<j+1$,
we have 
\begin{equation*}
  -N'(t)=n_j-n_{j+1}=\ceil{n_jk/n}\ge\ceil{kN(t)/n}.
\end{equation*}
If $t_{i}<t<t_{i-1}$, then $kN(t)/n>i-1$ and thus
$-N'(t)\ge i$. Hence,
\begin{equation*}
  i(t_{i-1}-t_i) \le \int_{t_i}^{t_{i-1}} -N'(t)\,dt =
  N(t_{i})-N(t_{i-1})
=\frac nki - \frac nk(i-1)
=\frac nk.
\end{equation*}
Consequently,
using \eqref{g2a} and
noting that $t_0=\min\set{t:N(t)=0}$ is an integer by
the definition of $N$,
\begin{equation*}
\nn(S,G)
\le
\nngr(S,G)
\le
  t_0
=\sum_{i=1}^k(t_{i-1}-t_i)
\le \sum_{i=1}^k\frac n{ki}
=\frac nk H_k.
\qedhere
\end{equation*}
\end{proof}

\begin{remark}
Lorentz~\cite{Lorentz} applied the greedy algorithm in a slightly different 
context (coverings of the natural numbers by translates of an infinite set);
he remarked that his method applies (in our terminology) to bound $\nn(S,\bbzn)$
when $|S|=k$, although he only stated the weak bound $O(n\log k/k)$ for this case.

Newman~\cite{Newman_dens} proved a result related to Corollary~\ref{Cgreedy} 
but for coverings of the natural numbers, or equivalently of $\bbZ$:
he showed that if $S\subset \bbZ$ with $|S|=k$,
then $\kk(S,\bbZ)$ (which we define in \refS{SZ}) satisfies
$\kk(S,\bbZ)\le \log k+1$. His method (picking translates randomly
to cover most of $G$ and then using one translate for each remaining element)
applies just as well to subsets of a finite group, but gives
the slightly worse bound
$\kk(S,G)\le \min\{ks/n+k\exp(-ks/n)\}$
(corresponding to $\nn(S,G)\le \min\{s+n\exp(-ks/n)\}$),
where the minimum is over integer $s$.
Note that the minimum over real $s$ is $\log k+1$, attained at
$s=(n/k)\log k$.
\end{remark}

\begin{remark}\label{RJohnson}
There is an interesting connection between the covering problem
considered here, and the algorithmic problem \texttt{SET COVER};
we are grateful to a referee for bringing this to our attention.
In the \texttt{SET COVER} problem,
the input is a family $\cS=\{S_i\}$ of sets covering some ground
set $X$ (i.e., with union $X$), and the task is to find efficiently
a subcover of close to minimum size. Johnson~\cite{Johnson1}
showed that the greedy algorithm (pick sets one by one, choosing
one that includes the maximum number of uncovered points) achieves
an approximation ratio of at most $H_k$ when $|S_i|\le k$ for all $i$.
The proof is very different from that above: he considers the weight
function $w(U)=\sum_{S_i\in \cF} H_{|S_i\cap U|}$, where $\cF\subseteq\cS$
is a family covering $X$ with $|\cF|$ minimal. Let
$U\subseteq X$ denote the set of points as yet uncovered at some
stage. If the greedy algorithm next chooses a set $S_j$ covering
$r$ points of $U$, then $|S_i\cap U|\le r$ for all $S_i\in \cF$.
Each newly covered point decreases $|S_i\cap U|$ by one
for (at least) one $S_i\in \cF$, so the weight function
of the uncovered points decreases by at least $r\min\{H_a-H_{a-1}: a\le r\}=r/r=1$.
(Several different points may decrease the same $|S_i\cap U|$, but this is no problem.)
Hence the algorithm uses at most $w(X)=\sum_i H_{|S_i|}\le H_k|\cF|$ sets.

A simple modification of this argument gives an alternative proof of \refC{Cgreedy}.
Suppose $\cS$ is a cover of $X$, and $\cF\subseteq\cS$ is an $\ell$-cover of $X$,
so each $x\in X$ is in at least $\ell$ sets in $\cS$. Taking $w(U)=\sum_{S_i\in \cF} H_{|S_i\cap U|}$
as before, when running the greedy algorithm to find a ($1$-)cover of $X$,
the weight of the uncovered points decreases by at least $\ell$ at every
step, since each newly covered point is in at least $\ell$ sets $S_i\in\cF$.
If all $S_i$ have size at most $k$, the greedy algorithm
thus uses at most $w(X)/\ell=H_k|\cF|/\ell$ sets. Taking $X=G$, $\cS=\{t+S:t\in G\}$,
$\ell=k=|S|$ and $\cF=\cS$, this gives \eqref{g2f}.
\end{remark}

\begin{remark}
Together, \eqref{njdef} and \eqref{g2a} give the best bounds obtainable
by the simple argument in the proof of \refT{Tgreedy1}. However, this implicit form is not
very useful, which is why we give the explicit (inexact) estimates for $n_j$
above.
Other estimates of $n_j$ yield other estimates of $\nn(S,G)$.

For example, noting that $n_j\le{n(1-k/n)^j}$, if
\begin{equation*}
j=\floor{\log n/|\log(1-k/n)|}+1
>
\log n/|\log(1-k/n)|,  
\end{equation*}
then $n_j<1$ and thus $n_j=0$. Hence, 
\eqref{g2a} implies
\begin{equation}
    \nn(S,G)\le\frac{\log n}{|\log(1-k/n)|} +1. \label{g2b}
\end{equation}
Alternatively, taking
$
j=\ceil{\log k/|\log(1-k/n)|}
$,
we similarly obtain
\begin{equation*}
  n_j\le n(1-k/n)^j \le n/k;
\end{equation*}
thus at most $n/k$ further steps are needed and $\nngr(S,G)\le j+n/k$,
which yields 
  \begin{align}
  \nn(S,G)&\le\frac{\log k}{|\log(1-k/n)|}+\frac nk +1. \label{g2d}
\end{align}
If $k$ is comparable with $n$ (for example, $k=n/2$), then \eqref{g2b} and \eqref{g2d}
are better than \eqref{g2f}.
However, we are usually interested in $n$ much larger than $k$, or
even $n\to\infty$ with $k$ fixed. Then \eqref{g2f} and \eqref{g2d} give similar bounds
(with the former much cleaner), and \eqref{g2b} is much weaker.
\end{remark}

Returning to trivialities, we next note that one can bound the efficiency 
of the product of two sets in terms of those of the sets themselves.
\begin{lemma}\label{Lprod}
Let $S_1\subseteq G_1$ and $S_2\subseteq G_2$, where $G_1$ and $G_2$
are finite groups. 
Then
\begin{equation*}
 \eff(S_1,G_1)\eff(S_2,G_2) \le \eff(S_1\times S_2,G_1\times G_2) \le \min\{\eff(S_1,G_1),\eff(S_2,G_2)\}
\end{equation*}
or, equivalently,
\begin{equation}\label{kprod}
 \kk(S_1,G_1)\kk(S_2,G_2) \ge \kk(S_1\times S_2,G_1\times G_2) \ge \max\{\kk(S_1,G_1),\kk(S_2,G_2)\}.
\end{equation}
\end{lemma}
\begin{proof}
It is more convenient to prove the second form \eqref{kprod}. For the first inequality, we simply
take the product of two coverings: 
choose $T_i\subseteq G_i$ with $T_iS_i=G_i$ and $|T_i|=\nn(S_i,G_i)$.
Then $(T_1\times T_2)(S_1\times S_2)=(T_1S_1)\times (T_2S_2)=G_1\times G_2$,
so $\nn(S_1\times S_2,G_1\times G_2)\le |T_1\times T_2|=\nn(S_1,G_1)\nn(S_2,G_2)$,
which gives the inequality.

For the second inequality, consider any covering of $G_1\times G_2$ by translates of $S_1\times S_2$.
For each $g\in G_1$, this induces a covering of the copy $\{g\}\times G_2$ of $G_2$
by translates of $S_2$. By definition of $\kk(S_2,G_2)$, each element is covered
on average at least $\kk(S_2,G_2)$ times. Averaging over $g$, we see that in
the covering of $G_1\times G_2$, each element is on average covered at least $\kk(S_2,G_2)$
times. Hence $\kk(S_1\times S_2,G_1\times G_2) \ge \kk(S_2,G_2)$. Similarly
$\kk(S_1\times S_2,G_1\times G_2) \ge \kk(S_1,G_1)$.
\end{proof}

Of course, \refL{Lprod} applies just as well in the compact setting,
with suitable modifications to the proof.
Unsurprisingly, the inequalities may be strict. 
For example, let $G_1=G_2$ be the cyclic group with 5 elements and
let $S_1=S_2$ be a subset of $G_1$ of size $4$. It is easy to see
that $\nn(S_i,G_i)=2$, while $\nn(S_1\times S_2,G_1\times G_2)=3$.
This gives $\kk(S_i,G_i)=8/5$ and $\kk(S_1\times S_2,G_1\times G_2)=48/25$.
We shall give a less trivial example where the first inequality
in \eqref{kprod} is strict in \refE{Eprod}.

\section{Random $S$ in a finite group}\label{Srandom}

In this section we shall show that the bounds in \refC{Cgreedy}
are close to best possible, in that the worst case efficiency for
sets $S$ of size $k$ really is as small as $(1+o(1))/\log k$.
Newman~\cite{Newman_dens} proved such a result for subsets of
the natural numbers
(or, equivalently, of $\bbZ$, or of $\bbzn$ with $n$ large), using a random construction;
although his argument adapts easily to our setting, we
shall give a different proof that seems to generalize more easily.

In fact, we shall show a little more: if $n=n(k)$ grows at a suitable
rate, with $n/k$ tending to infinity but not too fast, then almost
any $k$-element subset of any given $n$-element group
covers inefficiently. This contrasts with our results in~\refS{SZsmall},
where we show that in the cyclic case, if $n$ grows
very rapidly with $k$, then almost all $k$-element subsets
cover with close to optimal efficiency.

\begin{theorem}\label{finite_bad}
Given $0<\delta<1$, for all sufficiently large $k$ there exists a
finite group $G$ and a subset $S$ of size $k$ with $\kk(S,G)\ge
(1-\delta)\log k$. 
In particular, $G$ can be chosen as any group of order $\floor{k\log k}$,
for example a cyclic group.
\end{theorem}

\begin{proof}
Pick positive integers $n=n(k)$ and $t=t(k)$ so that $n/k\to\infty$,
\begin{equation}\label{tkn}
 tk/n \sim (1-2\delta/3) \log k
\end{equation}
as $k\to\infty$, and $t\le k^{\delta/6}$ for sufficiently large $k$.
For example, we may take $n=\floor{k\log k}$  
and $t=\floor{(1-2\delta/3)(\log k)^2}$.

Let $G$ be any group of order $n$.
Set $p=k/n=o(1)$, and let $S$ be a random subset of $G$ obtained by selecting
each element independently with probability $p$, so $|S|$ has a binomial
distribution with mean $k$.

We shall show that for any $T\subset G$ with $|T|=t$, the probability
that $TS=G$ is  $o(1/\binom{n}{t})$ as $k\to\infty$,
uniformly in $T$. It follows that the expected
number of `good' sets $T$ for a given $S$ is $o(1)$, so with probability
$1-o(1)$, there is no good set for $S$. Since $\Pr(|S|\ge k)> 1/2$,
for $k$ large enough it follows that with positive probability, $|S|\ge k$ and $\tau(S,G)> t$.
In particular, deleting some elements if necessary, there is some $S$
with $|S|=k$ and $\tau(S,G)>t$, and the result follows.

To carry out this plan, fix a set $T\subset G$ with $|T|=t$.
For $y\in G$, let $I_y$ denote the indicator function of the event
that $y$ is \emph{not} in $S$, so $\E(I_y)=1-p$.
For $x\in G$, let $U_x$ denote the indicator function
of the event that $x$ is not covered by $TS$.
Then 
\begin{equation}\label{udef}
 U_x=\prod_{z\in T} I_{z^{-1}x},
\end{equation}
so $\E(U_x)=(1-p)^t$.
Finally, let $N=\sum_x U_x$ be the number of uncovered points.
Then, for $k$ large enough,
\begin{multline}\label{42b}
 \mu \= \E(N) = n(1-p)^t = n \exp(-(1-o(1))tk/n)\\
 \ge n\exp(-(1-\delta/2)\log k) = k^{\delta/2}(n/k).
\end{multline}
Let us write $x\sim y$ if the indicator functions $U_x$ and $U_y$ are
dependent. 
From \eqref{udef}, we have $x\sim y$ if and only if there are $z,z'\in T$
such that $z^{-1}x=(z')^{-1}y$, i.e., such that $z'z^{-1}=yx^{-1}$.
In other words $x\sim y$ if and only if $yx^{-1}\in TT^{-1}=\{z'z^{-1}: z,z'\in T\}$.
Hence, for each $x$, the number of $y$ such that
$x\sim y$ is $|TT^{-1}|\le t^2$.
Let
\begin{equation}\label{Dest}
 \overline\Delta = \sum_{x\sim y} \E(U_xU_y) \le \sum_{x\sim y}\E(U_x) \le t^2\sum_x \E(U_x)=t^2\mu,
\end{equation}
where the sum runs over all ordered pairs $x\sim y$, including those
with $x=y$.
By Janson's inequality \cite[Theorem 2.18]{JLR},
\begin{equation}\label{n0}
\P(TS=G)=
 \Pr(N=0) \le \exp\left(-\frac{\mu^2}{\overline\Delta}\right) 
\le \exp\left(-\frac{\mu^2}{t^2\mu}\right) = \exp(-\mu/t^2).
\end{equation}
Recalling that $t\le k^{\delta/6}$, from \eqref{42b} we have $\mu/t^2\ge k^{\delta/6}(n/k)$.
For $k$ sufficiently large we have
\begin{equation*}
 \binom{n}{t}\le \left(\frac{en}{t}\right)^t \le k^t = \exp(t \log k)
 \le \exp\Bigl(\frac nk(\log k)^2 \Bigr).
\end{equation*}
For large $k$, $(\log k)^2$ is much smaller than $k^{\delta/6}$, and
it follows that
\begin{equation*}
\sum_{|T|=t}\P(TS=G)
\le
 \binom{n}{t} \exp\Bigpar{-\frac{\mu}{t^2}}
\le \exp\Bigl(\frac nk(\log k)^2 - \frac nk k^{\delta/6}\Bigr) =o(1),
\end{equation*}
and the result follows.
\end{proof}

\begin{remark}\label{Rmore}
The proof above shows that for
$n(k)$ in a certain range, namely $n=k^{1+o(1)}$ with $n/k\to\infty$,
almost all $k$-elements subsets $S$ of \emph{any}
group of order $n$ have $\kk(S,G)\ge (1-\delta)\log k$, 
in the  sense that the proportion of such sets tends to 1 as $k\to\infty$.
(See \refT{unbal} below for an even stronger result.)
As written, the argument extends to larger $n(k)$ only if we are willing to accept
a decrease in the claimed inefficiency. However, we were rather careless in the proof,
simply estimating $\E(U_xU_y)$ by $\E(U_x)$.
Indeed, one in fact has
\begin{equation*}
 \E(U_xU_y) = (1-p)^{|T^{-1}x\cup T^{-1}y|} = (1-p)^{2t-|T^{-1}x\cap T^{-1}y|} = (1-p)^{2t-r(yx^{-1})},
\end{equation*}
where $r(yx^{-1})$ is the number of ways of writing $yx^{-1}$ as $z'z^{-1}$
with $z'$, $z\in T$. We have effectively bounded $r(yx^{-1})$ by $t$ for all $yx^{-1}$, whereas
one could consider how often $r(yx^{-1})$ can be large. In doing so, one may have
to consider special sets $T$, such as (in the Abelian case)
those containing large arithmetic progressions,
separately: one needs the sum over $T$ of the probability that $TS=G$
to be small, rather than a uniform bound for each~$T$.  

Note that there is a limit to how rapidly $n$ can be allowed to grow
with $k$; we shall show in \refS{SZsmall} that if
$n$ grows very rapidly with $k$, or if $k$ is fixed and $n\to\infty$,
then in the cyclic group $\bbZ_n$, for any $\delta>0$ almost all $k$-element subsets 
have $\kk(S,\bbZ_n)\le 1+\delta$. In general, it seems to be an interesting
(if in this form rather vague) question
to determine the `typical' value of $\kk(S,G)$ for a random $k$-subset
of a (perhaps specifically cyclic) group of order $n$, with $k=k(n)$ a function
of $n$.
\end{remark}

\begin{remark}
It is easy to adapt an argument of Newman~\cite{Newman_dens}
to give an alternative proof of \refT{finite_bad}. Newman 
makes the nice observation (Lemma 1 in~\cite{Newman_dens}) that if
$A_1,\ldots,A_r$ are $a$-element subsets of some groundset $G$ such
that no element of $G$ appears in more than $b$ of the $A_i$, and
$G_p$ is the random set formed by selecting each element
of $G$ independently with probability $p$, then the probability
that $G_p$ meets every $A_i$ is at most $(1-(1-p)^a)^{\ceil{r/b}}$.
Taking $A_x=T^{-1}x$ for each $x\in G$, and applying this with $r=n$
and $a=b=t$, one obtains the stronger bound
\begin{equation*}
 \Pr(TS=G) \le (1-(1-p)^t)^{n/t} \le \exp(-\mu/t)
\end{equation*}
in place of \eqref{n0}.
The reason for writing the argument as we did is that
we shall need an extension, \refT{unbal} below; our proof
of \refT{finite_bad} adapts easily to give this, while Newman's
argument does not seem to.
\end{remark}

\begin{remark}
Returning to the \texttt{SET COVER} problem mentioned in \refR{RJohnson},
a referee has raised the interesting question of whether there
is a connection between (the proof of) \refT{finite_bad} and inapproximability
results for \texttt{SET COVER}, in particular the result of Feige~\cite{Feige}
giving a lower bound of $(1+o(1))\log k$ for the approximation ratio
modulo certain complexity theoretic assumptions. Here we have no answers.
\end{remark}

Perhaps the main interest of \refT{finite_bad} is that 
it shows, together with \refC{Cgreedy},
that the overall `worst-case' \cd{} $\kk(S,G)$ for sets of size
$k$ is asymptotically $(1-o(1))\log k$.
In fact, calculating slightly more carefully in the proof of \refT{finite_bad},
taking $n=\floor{k\log k}$ as before and choosing $t$ so that 
$tk/n= \log k-7\log\log k+O(1)$,
we can find $k$-element
sets with $\kk(S,G)\ge \log k-O(\log\log k)$, giving the following result.
\begin{theorem}\label{worstk}
Let $\worst_k$ denote the supremum of $\kk(S,G)$ over all $k$-element 
subsets $S$ of all finite groups $G$. Then
\begin{equation*}
 \log k-O(\log\log k) \le \worst_k\le \log k+1.
\end{equation*}
\vskip-\baselineskip
\nopf
\end{theorem}
\refR{Rmore} suggests that the lower bound above can be improved, leading
to the following question. 
\begin{question}\label{qworst}
What is the order of the difference $\log k-\worst_k$? In particular,
is it true that $\worst_k=\log k-O(1)$?
\end{question}

In the proof of \refT{finite_bad}, we generated $S\subseteq G$
by selecting uniformly from \emph{all} elements of $G$.
Perhaps surprisingly, this is not essential: the following extension shows that we can
start from \emph{any subset} of $G$, as long as its size
is much larger than the number $k$ of elements we pick.

\begin{theorem}\label{unbal}
Let $0<\delta<1$ be given, and suppose that $n=n(k)$ and $h=h(k)$ are such that
$n\ge h$, $h/k\to\infty$, and $n\le k^{1+\delta/7}$.
If $k$ is large enough then,
given any group $G$ of order $n$ and any \emph{subset} $H$ of $G$ with $|H|=h$,
for most $k$-element subsets $S$ of $H$ we have $\kk(S,G)\ge (1-\delta)\log k$.
\end{theorem}

More precisely, as $k\to\infty$,
this holds for a fraction $1-o(1)$ of all $k$-element
subsets $S$ of $H$.

\begin{proof}
We modify the proof of Theorem~\ref{finite_bad} very slightly.
As before, choose $t=t(k)\sim (n/k)(1-2\delta/3)\log k$;
the assumption on $n$ ensures that $t\le k^{\delta/6}$ for $k$ large.

This time, set $p=k/h=o(1)$, and form $S$
by selecting elements of $H$ independently with probability $p$,
so $|S|$ again has a binomial distribution with mean $k$, and $\Pr(|S|\ge k)\ge 1/2$.

For $y\in H$ let $I_y$ be the indicator function of the event
that $y\notin S$, and, as before,
for $x\in G$, let $U_x$ denote the indicator function
of the event that $x$ is not covered by $TS$.
This time
\begin{equation*}
 U_x=\prod_{z\in T\: :\: z^{-1}x\in H} I_{z^{-1}x},
\end{equation*}
so $\E(U_x)=(1-p)^{a_x}$, where $a_x=|H\cap T^{-1}x|$.
Note that
\begin{equation*}
 \sum_x a_x=|\{(x,z):x\in G, z\in T, z^{-1}x\in H\}| = |\{(y,z):y\in H,z\in T\}| = th.
\end{equation*}
With $N=\sum_x U_x$ the number of uncovered points as before, by convexity we have
\begin{equation*}
  \begin{split}
\mu &= \E(N) = \sum_x (1-p)^{a_x} \ge n(1-p)^{\sum_x a_x/n} =
n(1-p)^{th/n}
\\
& = n\exp(-(1-o(1))pth/n).
  \end{split}
\end{equation*}
Since $pth/n=tk/n$, this yields the same lower bound \eqref{42b} that we obtained in the
proof of \refT{finite_bad}. 
The estimate \eqref{Dest} is valid as is, so the rest of the proof of \refT{finite_bad}
carries over unmodified.
\end{proof}

We close this section by noting a simple consequence of \refT{finite_bad}.
\begin{example}\label{Eprod}
By \refT{finite_bad}, for every sufficiently large $k$ there exists an $n$
and a set $S\subset\bbZ_n$ with $|S|=k$ such that $\kk(S,\bbZ_n)\ge \log k/2$.
Consider the product set $S\times S\subset \bbZ_n\times \bbZ_n$.
By \refC{Cgreedy}, if $k$ is large enough, then
\begin{equation*}
 \kk(S\times S,\bbZ_n\times \bbZ_n) \le \log|S\times S|+1 
=2\log k+1 < (\log k/2)^2 \le \kk(S,\bbZ_n)^2.
\end{equation*}
This gives another example where the first inequality in \eqref{kprod} is strict.
\end{example}

\section{$G=\bbZ$}\label{SZ}

In this section we consider coverings of $G=\bbZ$, which is perhaps
the simplest non-compact group. 
In 1954, Lorentz~\cite{Lorentz} and Erd\H os~\cite{Erdos} considered coverings 
of $\bbZ$ by translates of an infinite subset $S$;
here we shall only consider coverings by 
finite, non-empty subsets $S\subset\bbZ$, first
studied by Newman~\cite{Newman_dens}. (Actually, all three papers
concern coverings of the natural numbers rather than $\bbZ$; this makes no difference.)
In this case, trivially,
$\nn(S,\bbZ)=\infty$, so to make sense of the notion
of covering efficiency we need to modify the definitions somewhat.

Let $\nn(S,n)$ be the smallest number of translates of $S$ that cover
the set $[n]\=\set{1,\dots,n}$, \ie,
\begin{equation}\label{tzn}
  \nn(S,n)\=\min\set{|T|:T+S\supseteq [n]}.
\end{equation}
Obviously, the number of translates required to cover any other
interval of integers of length $n$ is the same.
It follows immediately that $\nn(S,n)$ is subadditive:
\begin{equation*}
\nn(S,m+n)\le\nn(S,n)+\nn(S,m)  
\end{equation*}
for all $m,n\ge1$.
By a well-known result,
this implies the existence of the limit
\begin{equation}\label{tz}
  \nn(S)\=\lim_\ntoo \nn(S,n)/n
\end{equation}
and the equality
\begin{equation}\label{tzinf}
  \nn(S)=\inf_{n\ge1} \nn(S,n)/n.
\end{equation}
We call $\nn(S)$ the \emph{covering density} of $S$.
Newman~\cite{Newman_dens} 
used the term `codensity' for the minimum density $c(S)$
of a set $T\subseteq \bbZ$ with a defined density such that $T+S=\bbZ$.
It is easy to see that $\nn(S)$ and $c(S)$ are equal;
our more concrete definitions will be useful later.
(Schmidt and Tuller \cite{SchmidtTullerI,SchmidtTullerII} call $c(S)$
the `minimal covering frequency', and use
`minimal covering density' for our $\kk(S)$ below.) 

In analogy with \eqref{ksg}, we define
the \emph{finite \cd{}}
\begin{align}
  \label{kzn}
\kk(S,n)&\=\frac{\nn(S,n)|S|}{n},
\intertext{which, for large $n$, is essentially
the average number of translates covering each point in an optimal
covering of $[n]$, the \emph{(asymptotic) \cd}}
\label{kz}
  \kk(S)&\= \lim_\ntoo\kk(S,n)=\nn(S)|S|,
\intertext{and the \emph{efficiency} of $S$}
  \eff(S)&\=1/\kk(S).
\label{ez}
\end{align}
We sometimes use the notation $\nn(S,\bbZ)$ \etc{} for emphasis and clarity.

For a finite, non-empty $S\subset\bbZ$ we trivially have
\begin{equation}
\frac{n}{|S|}\le \nn(S,n) \le n
\end{equation}
and thus 
\begin{align}
\xfrac{1}{|S|}&\le \nn(S) \le 1,
\\
1&\le \kk(S) \le |S|,
\\
|S|\qw&\le \eff(S) \le 1.
\end{align}
We shall improve these bounds in \refT{TZgreedy} below.
Again, there are
many examples of \emph{efficient} sets $S$ with $\kk(S)=\eff(S)=1$,
as in \refE{Eresidue}, but we are mainly interested in the others.

Let $\ga_k$ be the minimal efficiency over $k$-element subsets
of $\bbZ$, i.e.,
\begin{equation}\label{ga}
\ga_k\=
\inf\set{\eff(S,\bbZ):|S|=k}.
\end{equation}
Note that $k\qw\le\ga_k\le1$, where the lower bound will be
improved later.
It is an interesting problem to find $\ga_k$ exactly for small
$k$.
Trivially, $\ga_1=1$. It is easy to see that $\ga_2=1$ too.
Indeed, suppose that $S\subset\bbZ$ with $|S|=2$. Translating
if necessary, we may assume that
$S=\set{0,a}$ with $a\ge1$. Then \set{S+i:0\le i<a} is a
partition of \set{0,\dots,2a-1}, and thus $\cT\=\set{S+i+2aj:0\le
  i<a,\;j\in\bbZ}$ is a partition of $\bbZ$.
It follows that $\eff(S)=1$.
Hence, $\ga_2=1$.

Newman~\cite{Newman_dens} showed that $\ga_3=5/6$, and 
(according to~\cite{Weinstein}) conjectured
that $\ga_4=3/4$; Weinstein~\cite{Weinstein} showed that $\ga_4\le 3/4$,
and used a computer to prove a lower bound of $0.7354...$
Checking all subsets $S$ of a small interval
suggests the values $\ga_4=3/4$, $\ga_5=11/5$ and $\ga_6=2/3$; see \refR{remarkvals}.

We may talk about the efficiency \etc{} of a particular covering
$\cT=\set{t_i+S}$ of $\bbZ$ or $\bbZ^+$ too.
We define
\begin{align}
\nn(\cT,n)&\=|\set{i:0\le t_i<n}|,
\label{ttzn}
\intertext{and, provided the limit exists, the \emph{covering density}}
  \nn(\cT)&\=\lim_\ntoo \nn(\cT,n)/n,
\\
\intertext{the \emph{covering multiplicity}}
  \kk(\cT)&\= \nn(\cT)|S|,
\\
\intertext{and the \emph{covering efficiency}}
  \eff(\cT)&\=\frac1{\kk(\cT)}
=\frac1{\nn(\cT)|S|}
\label{etz}
.
\end{align}
Actually, we shall only use these notions when $\cT$ is periodic; in
this case the limit exists and, if $\cT$ has period $\ell$, then
\begin{equation}
  \nn(\cT)=\nn(\cT,\ell)/\ell.
\label{ttzl}
\end{equation}

From the definitions above, it is not immediately clear that there is
an infinite covering $\cT$ with $\kk(\cT)=\kk(S)$. However, the reader
will not fall off her chair on learning that this is indeed the
case. In fact, unsurprisingly, there is always an optimal periodic
covering; see, for example, Schmidt and Tuller~\cite{SchmidtTullerI}. 
We shall give a different proof that gives a better bound
on the period, and, perhaps more importantly, leads to a practical
algorithm for calculating $\kk(S)$.

Suppose without loss of generality that $S\subseteq [0,s]$, $s\ge 1$, with
$0\in S$.
Suppose that for some $N$, which we think of as large,
$T\subset \bbZ$ is such that $T+S$ covers $[N]$, with $|T|$ minimal,
i.e., $|T|=\nn(S,N)$. Note that $\min T\ge -s$.
For $0\le m\le N$, let $T_m=\{t\in T: t\le m\}$.
Then $T_m+S$ covers $[m]$, and $\max(T_m+S)\le m+s$.
Let $A_m=(T_m-m+s)\cap [s]$, so $A_m$ records which of the last $s$ possible
elements of $T_m$ are in fact present. 
Writing $\sigma$ 
for the $0/1$-sequence corresponding to $T$,
note that each $A_m$ corresponds to a block of $s$ consecutive terms of $\sigma$.

Passing from $m$ to $m+1\le N$, there are two possibilities.
In the first case $m+1\notin T$, so $T_{m+1}=T_m$ and $A_{m+1}=(A_m-1)\setminus\{0\}$.
In this case $S+T_m=S+T_{m+1}$ must cover $m+1$, so we must have $s+1\in A_m+S$.
In the second case, $m+1\in T$, so $A_{m+1}=(A_m-1)\setminus\{0\}\cup\{s\}$.

Let $\GS$ be the weighted directed graph defined as follows. For the vertex
set $V$ we take the power set $\Ps([s])$ of $\{1,2,\ldots,s\}$.
Send a directed edge from $A$ to $B$ with weight $0$ if and only if
$s+1\in A+S$
and $B=(A-1)\setminus\{0\}$, and send
a directed edge from $A$ to $B$ with weight $1$ if $B=(A-1)\setminus\{0\}\cup\{s\}$.
Given a (directed) walk $\gamma$ in $\GS$, let us write $|\gamma|$ for its length
(number of edges), and $w(\gamma)$ for its weight, i.e., the sum of the weights
of its edges.
Then  $\gamma = A_0A_1\cdots A_N$ is a walk in $\GS$
with $|\gamma|=N$ and $w(\gamma)=\sum_{m=0}^{N-1} |T_{m+1}\setminus T_m|=|T\setminus T_0|$.
Since $\min T\ge -s$, we have $w(\gamma)=|T|+O(1)$.

The \emph{de Bruijn graph} associated to binary sequences
of length $s\ge 1$ is the 2-in, 2-out directed graph
whose vertices are binary sequences of length $s$,
with an edge from $A$ to $B$ if the last $s-1$ terms of $s$ agree
with the first $s-1$ terms of $B$, i.e., if $A$ and $B$
can appear as subsequences of $s$ consecutive terms
in a single (infinite) binary sequence $\sigma$, with $B$ starting
one place later than $A$. This graph
was introduced independently by de Bruijn~\cite{deBruijn} and Good~\cite{Good}
in 1946. It is well known and easy to check
that directed $n$-edge walks in the de Bruijn graph
correspond bijectively to binary sequences of length $s+n$.

Identifying subsets of $[s]$ with binary sequences of length $s$
in the usual way,
the unweighted directed graph underlying $\GS$ is a subgraph of the de Bruijn graph.
It is not hard to check that, given a walk $\gamma$ in $\GS$ with $|\gamma|=N$,
one can construct a set $T$ of size at most $w(\gamma)+s$ so that $T+S$
covers $[N]$. We omit the details since we only need the cycle
case, which we prove below. As usual, by a \emph{cycle}
in a directed graph $G$ we mean a directed walk $v_0v_1\cdots v_{n-1}v_0$
in $G$ in which the vertices $v_0,\ldots,v_{n-1}$ are distinct.

\begin{theorem}\label{T:rat}
Let $\GS$ be the directed graph defined above, and let $\alpha=\alpha(S)$
be the minimum of $w(\gamma)/|\gamma|$ over all cycles $\gamma$ in $\GS$.
Then $\nn(S)=\alpha$, and there is a periodic covering $\cT$ with
$\nn(\cT)=\nn(S)$. In particular, $\nn(S)$ is rational
for every finite, non-empty $S\subset\bbZ$.
\end{theorem}

\begin{proof}
Suppose first that $T+S\supseteq [N]$,
with $|T|$ minimal. We have shown that there is a walk $\gamma$
with $|\gamma|=N$ and $w(\gamma)=|T\setminus T_0|\le |T|$.
Now $\gamma$ may be written as a union of cycles together with a union of vertex
disjoint paths. It follows that $w(\gamma)\ge \alpha(N-|V|)\ge \alpha N-2^s$.
Hence $\nn(S,N) \ge \alpha N -2^s$. Letting $N\to\infty$, we see that
$\nn(S)\ge \alpha$.

Conversely, let $\gamma=A_0A_1\dots A_\ell$ be a cycle in $G_S$ with length
$\ell$ and weight $w(\gamma)=\alpha \ell$. Set $A_{k+\ell m}=A_k$ for
all $0\le k<\ell$ and $m\in \bbZ$,
and let $T$ be the set of $k\in \bbZ$ such that $A_kA_{k+1}$
has weight $1$. Then $(A_k)$ is a 2-way infinite walk in $\GS$, and hence
in the de Bruijn graph, and each $A_k$ codes the intersection of $T$
with the corresponding interval of length $s$.
For each $k$, the fact that $A_kA_{k+1}\in E(\GS)$ ensures that $T+S$ covers $k+1$.
Clearly, $\cT=\{t+S:t\in T\}$
is periodic with density $\ga$,
so we have $\nn(S)\le \nn(\cT)=\alpha$, completing the proof.
\end{proof}

Let $\ell(S)$ denote the minimum period of an optimal periodic 
covering of $\bbZ$ by translates of $S$,
and let $\ell(s)$ denote the maximum of $\ell(S)$ over
all sets $S\subseteq [0,s]$, i.e., over all subsets
of $\bbZ$ with diameter at most $s$.
Schmidt and Tuller~\cite{SchmidtTullerI} show that $\ell(s)\le s(2^s+1)$,
and state that an improvement of this bound would be of interest.
The proof of \refT{T:rat} shows that $\ell(S)\le |G_S|=2^s$; in fact,
we can show a little more.

\begin{theorem}\label{Tperiod}
Suppose that $S\subseteq [0,s]$. Then
\begin{equation*}
 \ell(S)\le \min\Bigl\{ 2^s,\ 2\sum_{t\le 2s\nn(S)}\binom{s}{t} \Bigr\}.
\end{equation*}
Furthermore, there is a function $f(k)$ tending to $0$ as
$k\to\infty$ such that $\ell(S)\le 2^{sf(k)}$ for all $S\subseteq
[0,s]$ with $|S|=k$. 
\end{theorem}
\begin{proof}
Let $\cT$ be the periodic covering constructed in the proof of \refT{T:rat}, 
corresponding to the cycle $\gamma$ in $G_S$. Writing $w(v)$ for the \emph{weight} of a vertex $v$,
i.e., the fraction of $1$s present in the corresponding $0/1$-sequence of length $s$,
the average of $w(v)$ over the vertices in $\gamma$ is $w(\gamma)/|\gamma|
=\nn(\cT)=\nn(S)$. Hence, at least half the vertices have $w(v)\le 2\nn(S)$.
Thus the length $\ell(S)$ of $\gamma$ is at most twice
$\sum_{t\le 2s\nn(S)}\binom{s}{t}$, the number
of vertices with weight at most $2\nn(S)$.
Together with the trivial bound $\ell(S)\le |G_S|=2^s$, this proves the first statement.

Newman~\cite{Newman_dens} showed that if $|S|=k$, then
$\nn(S)\le (\log k+1)/k=o(1)$; see also \refC{Cgreedy} and \refT{TZgreedy} below.
Using this, the second statement of the theorem follows from the first, noting
that $\sum_{t\le 2s(\log k+1)/k}\binom{s}{t}\le 2^{sf(k)}$ for some $f(k)$ that
tends to $0$ as $k\to\infty$.
\end{proof}


It seems likely that the bounds in \refT{Tperiod}
are far from the truth. For tilings of $\bbZ$, 
i.e., for sets $S$ with $\kk(S)=1$, much better bounds are known. Indeed,
Bir\'o~\cite{Biro} showed that $\ell(S)\le \exp(s^{1/3+o(1)})$ in this case.
This suggests the following question, where we expect the answer to the second
part to be positive.

\begin{question}\label{q1}
What is the approximate growth rate of the function
$\ell(s)=\max\{\ell(S):S\subseteq [0,s]\}$? In particular,
is $\ell(s)=2^{o(s)}$?
\end{question}

\begin{remark}\label{Ralg}
Using \refT{T:rat} it is easy to give an algorithm for calculating $\nn(S)$.
Of course, the upper bound $2^s$ on $\ell(S)$ gives a trivial algorithm,
of complexity roughly $2^{2^s}$; considering the graph $G_S$ allows
one to reduce this greatly.
Indeed, given a vertex $v_0$ of $G_S$ and a positive integer $T$, using
breadth-first search one can easily find inductively for all $t\le T$
and for all $v$ the minimum weight of a walk from $v_0$ to $v$ of length
$t$ (if there is one). Since a cycle is a walk of length at most $|G_S|$,
considering each starting vertex separately one can thus determine 
the minimum average weight $\alpha$ of a cycle
in time at most $O(|G_S|^3)=O(2^{3s})$. This is practical for moderate $s$,
say $s\le 10$. In fact, one can do a little better by considering a different 
graph $G_S'$: the definition is similar to that of $G_S$, except
that a vertex now encodes $B_m=(T_m+S-m)\cap [s]$, corresponding
to the subset of $[m+1,m+s]$ covered by $T_m+S$, rather than $A_m=(T_m-m+s)\cap [s]$.
It is easy to see that $A_m$ determines $B_m$, but many possible $A_m$ may
lead to the same $B_m$, so $G_S'$ has fewer vertices than $G_S$.

The problem of finding a cycle with minimum average edge weight
in a given weighted directed graph $G$ has received considerable attention:
it is usually referred to as the MCMP or `Minimum Cycle Mean Problem'.
Suppose that $G$ has $n$ vertices and $m$ edges. Karp~\cite{Karp78} gave
a very elegant and easy to program algorithm with running time $O(nm)$,
which in the present context is $O(|G_S|^2)$. (This algorithm is very unintuitive,
and the proof that it is correct requires a little work.) 
More complicated but more efficient algorithms have since been found;
perhaps the most efficient is that of Orlin and Ahuja~\cite{OrlinAhuja},
with running time $O(n^{1/2} m \log n)$ when, as here,
the edge weights are integers that are not too large.
Having found the minimum average edge weight, one can then find
the shortest cycle with this average weight by using Johnson's algorithm~\cite{Johnson}.
\end{remark}
\begin{table}[ht]
\small
\begin{equation*}
\begin{array}{cc}
\begin{tabular}{|c|c|l|}
 \hline
  $s$ & $\ell(s)$ & $S$\\
  \hline
  0 & 1 & $\{0\}$ \\
  1 & 2 & $\{0,1\}$ \\
  2 & 4 & $\{0,2\}$ \\
  3 & 5 & $\{0,1,3\}$ \\
  4 & 8 & $\{0,4\}$ \\
  5 & 8 & \\
  6 & 13 & $\{0,1,4,6\}$ \\
  7 & 13 & \\
  8 & 27 & $\{0,1,3,8\}$ \\
  9 & 27 & \\ 
  10 & 45 & $\{0,2,7,10\}$ \\
\hline
\end{tabular}
&
\begin{tabular}{|c|c|l|}
 \hline
  $s$ & $\ell(s)$ & $S$\\
  \hline
  11 & 53 & $\{0,1,7,9,11\}$ \\
  12 & 66 & $\{0,1,8,12\}$ \\
  13 & 109 & $\{0,1,3,7,12,13\}$ \\
  14 & 129 & $\{0,1,12,14\}$ \\
  15 & 147 & $\{0,1,8,12,15\}$ \\
  16 & 147 & \\
  17 & 170 & $\{0,3,7,17\}$ \\
  18 & 192 & $\{0,2,3,7,18\}$ \\
  19 & 250 & $\{0,1,3,7,11,17,19\}$ \\
  20 & 286 & $\{0,1,4,10,12,19,20\}$ \\
  21 & 317 & $\{0,1,2,7,12,20,21\}$ \\ 
\hline
\end{tabular}
\end{array}
\end{equation*}
\medskip
\caption{The maximum over $S\subseteq [0,s]\subset \bbZ$ of the
  minimum period of an optimum covering 
by translates of $S$, with examples of sets $S$ achieving the maximum. When not listed, $S$ is as in the row above.}\label{tab_t1}
\end{table}

Using, for example, Karp's algorithm and then Johnson's algorithm,
one can fairly easily search over all $S$ with diameter
$s\le 21$, say, to obtain the information in Table~\ref{tab_t1}. Here the third column gives
an example of a set $S$ attaining the maximum. To us, sadly, this table does not strongly suggest any particular asymptotic form of $\ell(s)$!

\begin{remark}
Writing $\ell(s,k)$ for $\max\{\ell(S): S\subseteq [0,s],\,|S|=k\}$,
in the light of \refT{Tperiod}, to show that $\ell(s)=2^{o(s)}$ it
suffices to show that $\ell(s,k)=2^{o(s)}$ for each fixed $k$.
However, while the numerical data in Table~\ref{tab_k34} strongly suggests this holds
when $k=3$, even for $k=4$ the data is much less clear!
For $k=3$, Schmidt and Tuller~\cite{SchmidtTullerI} have conjectured a simple
formula for $\nn(\{0,a,b\})$ with $0<a<b$ and $a$ and $b$ coprime. If this formula
is correct, then $\ell(s,3)\le 2s$ for all $s$.
\begin{table}[htb]
\small
\begin{equation*}
\begin{array}{l}
\begin{tabular}{|c||c|c|c|c|c|c|c|c|c|c|c|}
  \hline
$s$         & 3 & 4 & 5 & 6  & 7  & 8  & 9    & 10 & 11   & 12 & 13\\
 \hline
$\ell(s,3)$ & 5 & 5 & 8 & 11 & 11 & 11 & 17   & 17 & 20   & 23 & 23\\
\hline
$\ell(s,4)$ &   & 7 & 8 & 13 & 13 & 27 & (26) & 45 & (35) & 66 &(58)\\
\hline 
\end{tabular}
\\
\\
\begin{tabular}{|c||c|c|c|c|c|c|c|c|c|}
  \hline
$s$         & 14  & 15 & 16  & 17  & 18  & 19  & 20  & 21  & 22  \\
 \hline
$\ell(s,3)$ & 23  & 29 & 29  & 32  & 35  & 35  &(32) & 41  & 41 \\
\hline
$\ell(s,4)$ & 129 &(91)& 122 & 170 & (95)&(100)& 183 &(143)& 185\\
\hline 
\end{tabular}
\end{array}
\end{equation*}
\medskip
\caption{Longest repeating periods for optimal periodic coverings of sets with size
$k=3$ and $k=4$ contained in $[0,s]$, found using Karp's algorithm. Values in brackets
are the longest for such $S$ with $0,s\in S$; in these cases there is a subset $S\subseteq [0,s-1]$
with longer period.}\label{tab_k34}
\end{table}
\end{remark}

\begin{remark}\label{remarkvals}
Returning to $\ga_k$, the infimum of $\eff(S,\bbZ)$ over $S$ with $k$ elements, we do
not have an algorithm to calculate this. The problem is that we have no bound
on the diameter $\diam(S)=\max S-\min S$ of a set achieving (or approaching) the infimum.
Again, searching small examples strongly suggests that this maximum diameter does
not grow too fast with $k$; see Table~\ref{tab_gak}.
\begin{table}[htb]
\small
\begin{tabular}{|c|c|l|}
 \hline
  $k$ & $\ga_k$ & $S$ \\
  \hline
  2 & $1$ & $\{0,1\}$  \\
  3 & $5/6$ & $\{0,1,3\}$  \\
  4 & $3/4$? & $\{0,1,2,4\}$  \\
  5 & $11/15$? & $\{0,1,3,4,8\}$  \\
  6 & $2/3$? & $\{0,1,2,4,6,9\}$ \\
 \hline
 \end{tabular}
 \medskip
 \caption{Known and conjectured values of $\ga_k$ for small $k$, with examples
of corresponding sets $S$; the values
with question marks are upper bounds, given by the minimum over
all $k$-element subsets $S$ of $[0,22]$.}\label{tab_gak}
\end{table}

\begin{question}\label{qvals}
Are the values in Table~\ref{tab_gak} correct? In particular, is $\ga_4$ equal to $3/4$,
as conjectured by Newman (see Weinstein~\cite{Weinstein})?
\end{question}
\end{remark}

Although it is not our focus here, let us remark that the question 
of which subsets $S$ \emph{tile} $\bbZ$ has received
much attention. Formally, $S$ tiles $\bbZ$ if there exists
a set $T$ such that every integer has a unique expression
as $t+s$, $t\in T$, $s\in S$. In the light of \refT{T:rat},
this is equivalent to $\kk(S)=1$. When $|S|$ is a prime power,
Newman~\cite{Newman_tess} gave a simple necessary and sufficient
condition
for $S$ to tile $\bbZ$; Coven and Meyerowitz~\cite{CovenMeyerowitz} answered the question
when $|S|$ has two prime factors. In the general case, there are only
partial results; see, for example, Konyagin and {\L}aba~\cite{KL}.

\section{Linear vs.\ cyclic}\label{Slc}

If $S\subset\bbZ$ is finite and $n$ is sufficiently large, say $n>\diam(S)$,
then we can regard $S$ as a subset of $\bbZ_n=\bbZ/n\bbZ$.
Clearly, if $T\subseteq\bbZ$ is a set such that $T+S\supseteq[n]$, and
$\hT$ is the image of $T$ in $\bbZ_n$, then $\hT+S=\bbZ_n$ and
$|\hT|\le|T|$. Thus
\begin{equation}\label{l1}
\nn(S,\bbzn)\le\nn(S,n).
\end{equation}
Equality does not necessarily hold, 
see \refE{El0} below,
but the difference is negligible for $S$ fixed and $n$ large, as shown by the
following lemma.

\begin{lemma}\label{LZn}
Let $S\subset\bbZ$ be a finite set and $n>\diam(S)$.
 Then 
\begin{equation}
  \label{l2a}
\nn(S,\bbZ)
\le
\frac{\nn(S,\bbzn)}{n}
\le
\frac{\nn(S,n)}{n}
\end{equation}
and
\begin{equation}
  \label{l2b}
\kk(S,\bbZ)
\le
{\kk(S,\bbzn)}
\le
\kk(S,n).
\end{equation}
More precisely, for any $m\ge1$, with $d=\diam(S)$,
\begin{equation}
  \label{l1b}
\nn(S,mn-d)\le m\nn(S,\bbzn).
\end{equation}
Furthermore, for any finite $S\subset\bbZ$,
\begin{align}
  \label{l2c}
\limn
\frac{\nn(S,\bbzn)}{n}
&
=\nn(S,\bbZ)
=
\inf_{n\ge1} \frac{\nn(S,\bbzn)}{n},
\\
  \label{l2d}
\limn {\kk(S,\bbzn)}
&=
\kk(S,\bbZ)
=
\inf_{n\ge1}
\kk(S,\bbzn),
\\
  \label{l2e}
\limn {\eff(S,\bbzn)}
&=
\eff(S,\bbZ)
=
\sup_{n\ge1}
\eff(S,\bbzn)
.
\end{align}
 \end{lemma}

\begin{proof}
Let $S$ be a finite subset of $\bbZ$ with diameter $d$. By translating
$S$, which does not change any of the numbers studied
here, we may and shall assume that $S\subseteq[0,d]$.
For $n>d$, regarding $S$ as a subset of $\bbZ_n$ as above,
let $\{\hat t+S:\hat t\in \hT\}$ be an optimal covering of $\bbZ_n$ by translates
of $S$. Choosing a representative $t\in [0,n-1]\subset \bbZ$ for each 
$\hat t\in \hT$, we obtain
a set $T\subseteq[0,n-1]\subset \bbZ$ with $|T|=|\hT|=\nn(S,\bbzn)$ such
that $T+S$ contains a representative of every residue class modulo $n$.
Since $S\subseteq[0,d]$, we have $T+S\subseteq[0,n+d-1]$.
Since $n>d$, this interval contains at most two representatives
of each residue class modulo $n$, and it follows that
$T+S\supseteq[d,n-1]$, and that if $x\in[0,d-1]$, 
then $x\in T+S$ or $x+n\in T+S$ (or both).
For $m\ge1$, let $T_m\=T+\set{jn:0\le j<m}$;
it follows then that 
$T_m+S\supseteq[d,mn-1]$. 
Consequently,
\begin{equation}
\nn(S,mn-d)
\le |T_m|
=m|T|
= m\nn(S,\bbzn),
\end{equation}
which proves \eqref{l1b}.

Dividing \eqref{l1b} by $mn-d$ and letting $m\to\infty$ we obtain the
first inequality in \eqref{l2a}; the second follows by \eqref{l1}.
Next, \eqref{l2a} implies \eqref{l2c} by \eqref{tz} and \eqref{tzinf}.
The remaining statements then follow from the definitions: 
multiplying \eqref{l2a} and \eqref{l2c} by $|S|$ we obtain \eqref{l2b}
and \eqref{l2d} by
\eqref{ksg}, \eqref{kzn}, \eqref{kz}.
Finally, \eqref{l2e} follows by taking reciprocals.
\end{proof}

 \begin{example}  \label{El0}
If $S=\set{0,1,5}$, then $\nn(S,\bbZ_6)=2$ because $S\cup(S+3)=\bbZ_6$
(see also \refE{Eresidue} below), but it is impossible to cover the
interval $[0,5]\subset\bbZ$ by two translates of $S$ and thus
$\nn(S,6)>2$.
(In fact, it is easily seen that $\nn(S,6)=3$.)
 \end{example}

\begin{example}\label{Eresidue}
If $S\subset\bbZ$ contains exactly one element from each
residue class modulo $k\=|S|$, then,
for every $m\ge1$, working in $\bbZ_{mk}$ we have
$\set{jk:0\le j<m}+S=\bbZ_{mk}$ and thus
$\nn(S,\bbZ_{mk})=m=|\bbZ_{mk}|/|S|$ 
and $\kk(S,\bbZ_{mk})=1$.
Consequently, by \refL{LZn},
$\nn(S)=1/|S|$ and $\kk(S)=\eff(S)=1$.	
  \end{example}

Our next result slightly improves Newman's result~\cite{Newman_dens}
for coverings of the natural numbers.

\begin{theorem}\label{TZgreedy}
  If $S\subseteq\bbZ$ is finite with $|S|=k\ge1$, then
\begin{align}
  \kk(S)&\le H_k\le \log k+1,
\\
  \eff(S)&\ge \frac1{H_k}\ge \frac1{\log k+1}.
\end{align}
In other words, $\ga_k$, the minimal efficiency of a $k$-set, is at least $1/H_k$.
\end{theorem}
\begin{proof}
  An immediate consequence of \eqref{l2d} and \eqref{g2fk}.
\end{proof}

\refL{LZn} shows that $\bbZ$ is `easier' to cover efficiently
than $\bbZ_n$. The next lemma shows that the difference is not that large.

\begin{lemma}\label{LZn2}
Let $S\subset \bbZ$ have diameter $d< \infty$. Then for any $n>d$ we have
\begin{equation*}
 \kk(S,\bbZ_n) \le \frac{n+d}{n} \kk(S,\bbZ) < 2\kk(S,\bbZ).
\end{equation*}
\end{lemma}
\begin{proof}
Without loss of generality, suppose that $\min S=0$, so $\max S=d$.
Let $\cT=\{t+S: t\in T\}$ be a covering with $\kk(\cT)=\kk(s,\bbZ)$.
By \refT{T:rat} such a $\cT$ exists, and we may assume that $\cT$ is periodic.
(We do not need this here, and could just as well work with an optimal
covering of $[m]$ for $m$ large.)

For $a\in \bbZ$ let $I_a=[a,a+n-1]$, an interval of length $n$.
Then $t+S$ meets $I_a$ if and only if $t\in J_a=[a-d,a+n-1]$, an interval
of length $n+d$. Hence, $T_a+S\supset I_a$, where $T_a=T\cap J_a$.
Since $\cT$ (which is periodic) has density $\nn(S,\bbZ)=\kk(S,\bbZ)/|S|$,
a simple averaging argument shows that there is some $a$
with $|T_a|\le |J_a|\kk(S,\bbZ)/|S| = (n+d)\kk(S,\bbZ)/|S|$.
In $\bbZ_n$ we have $\hT_a+S\supset \widehat{I_a}=\bbZ_n$,
so $\nn(S,\bbZ_n)\le (n+d)\kk(S,\bbZ)/|S|$ and the result follows.
\end{proof}

Combined with \refT{unbal}, \refL{LZn2} gives a new
proof of Newman's result~\cite{Newman_dens} concerning the asymptotic value
of $\ga_k$, the minimal efficiency of a $k$-set.
\begin{theorem}\label{Ta}
We have $\ga_k\sim 1/\log k$ as $k\to\infty$.
\end{theorem}
\begin{proof}
In the light of \refT{TZgreedy}, we need only prove that
$\ga_k\le (1+o(1))/\log k$.
Given $\delta>0$, set $h=h(k)=\ceil{k\log k}$ and $n=n(k)=\ceil{k(\log k)^2}$, say.
Then the assumptions of \refT{unbal} are satisfied. Applying this result
to $[h]\subset \bbZ_n$, it follows that for $k$ large enough,
there is some $S\subset [h]$ with $|S|=k$
such that $\kk(S,\bbZ_n)\ge (1-\delta)\log k$.
Since $\diam(S)\le h$ and $h/n\to 0$, from \refL{LZn2} it follows that
\begin{equation*}
 \kk(S,\bbZ) \ge \kk(S,\bbZ_n)\frac{n}{n+h} \ge (1-\delta-o(1))\log k.
\end{equation*}
Since $\delta>0$ is arbitrary, it follows that
$\sup\{\kk(S,\bbZ):|S|=k\} \ge (1-o(1))\log k$ as $k\to\infty$.
In other words, $\ga_k\le (1+o(1))/\log k$.
\end{proof}

\section{Small sets in $\bbZ$ and $\bbZ_n$}\label{SZsmall}

In this section we shall consider small sets $S$ in either 
$\bbZ$ or $\bbZ_n$ ($n$ large), 
with the aim of showing that for `typical' sets $S$ of some
fixed size $k$, the efficiency is close to 1.
Given an integer $x\in\bbZ$, we sometimes write $x$ for the corresponding
element in $\bbzn$, but in this section
for clarity we usually 
denote the latter by $\q x$. Thus $\q x=\q y \iff x\equiv y \pmod n$.

Schmidt~\cite{Schmidt} proved results relating
the \cd\ of $S$ to the linear equations satisfied 
by its elements, showing that if the elements satisfy no non-trivial
equations with small coefficients, then $S$ covers efficiently. 
Restated in our terminology, these results are as follows.
\begin{theorem}[\cite{Schmidt}]
  \label{TJ1}
For every $k\ge2$ and $\eps>0$ there exists a constant
$C=C(k,\eps)$ such that
for every $n\ge k$, if $S=\set{\q x_1,\dots,\q x_k}$ is a $k$-subset of
the cyclic group $\bbzn$ with $\kk(S,\bbzn)\ge 1+\eps$, then
there exist integer coefficients $a_1,\dots,a_k$, with
$\sum_{i=1}^k a_i=0$ and $0<\max_i |a_i|\le C$, such that 
\begin{equation}\label{TJ1c}
  \sum_{i=1}^k a_i\q x_i = 
\sum_{i=1}^{k-1} a_i(\q x_i-\q x_k) = 0
\end{equation}
in $\bbzn$.
\end{theorem}

\begin{theorem}[\cite{Schmidt}] \label{TJ1Z}
For every $k\ge2$ and $\eps>0$ there exists a constant
$C=C(k,\eps)$ such that
if $S=\set{x_1,\dots,x_k}$ is a $k$-subset of
$\bbZ$ with $\kk(S,\bbZ)\ge 1+\eps$, then
there exist integer coefficients $a_1,\dots,a_k$ with
$\sum_{i=1}^k a_i=0$ and $0<\max_i|a_i|\le C$, such that 
\begin{equation*}
  \sum_{i=1}^k a_i x_i = 0.
\end{equation*}
\end{theorem}

Although Schmidt proved his results above in a different order, it is
easy to see that \refT{TJ1} implies \refT{TJ1Z}.

\begin{proof}
Let $C=C(k,\eps)$ be as in \refT{TJ1}.
Given $S=\set{x_1,\dots,x_k}\subset\bbZ$, choose any integer $n>C\sumik|x_i|$
and let $\q S=\set{\qx_1,\dots,\qx_k}\subseteq\bbZ_n$. 
Then $\kk(\q S,\bbzn)\ge \kk(S,\bbZ)$ by \eqref{l2b}.
Hence \refT{TJ1} gives  $\sumik a_ix_i\equiv0\pmod n$ for some $a_i$
with $\sumik a_i=0$ and $0<\max_i |a_i|\le C$.
From the choice of $n$ it follows that $\sumik a_ix_i=0$.
\end{proof}

Theorems \ref{TJ1} and \ref{TJ1Z} easily imply that `typical' sets 
cover efficiently; let us make this precise. We consider $\bbzn$ first.
As usual, $\pto$ denotes convergence in probability.

\begin{corollary}
  \label{CJ1}
Let $S_{n,k}$ denote a $k$-element subset of $\bbzn$ chosen uniformly at random.
For every fixed $k$ we have $\eff(S_{n,k},\bbzn)\pto1$ as \ntoo.
\end{corollary}

\begin{proof}
With $k$ fixed, let $\tS=\tS_{n,k}\=\set{\qx_1,\dots,\qx_k}$, where $\qx_1,\dots,\qx_k$ are
independent uniformly random elements of $\bbzn$.
Then $S_{n,k}\eqd(\tS\mid|\tS|=k)$, and since $\Pr(|\tS|\neq k)\le k^2/n=o(1)$, it
suffices to consider $\tS$. 

Let $\eps>0$ and apply \refT{TJ1}. For every fixed
$(a_1,\dots,a_k)\neq(0,\dots,0)$, the sum $\sum_ia_i\qx_i$ is
uniformly distributed on $\bbzn$, and thus 
$\Pr(\sum_ia_i\qx_i=0)=1/n$.
Hence the probability that one of the at most $(2C+1)^k$ possible
conditions \eqref{TJ1c} holds
is at most $(2C+1)^k/n =o(1)$, and thus \refT{TJ1} yields
$\Pr(\kk(\tS,\bbzn)\ge1+\eps)\to0$ as \ntoo.
Thus $\kk(S,\bbzn)\pto1$. Equivalently, $\eff(S,\bbzn)\pto 1$.
\end{proof}

We have seen that, for fixed $k$ and large $n$, most $k$-subsets cover $\bbzn$ 
with close to optimal efficiency.
In contrast, \refT{finite_bad} and \refR{Rmore} show that if $n$ is
only slightly larger than $k$, then most $k$-subsets of $\bbzn$ have almost the
worst possible efficiency.
\begin{question}
Where is the threshold? 
More precisely, for which sequences $n=n(k)\ge k$ is it true that 
$\eff(S_{n,k},\bbzn)\pto1$ as \ktoo?
For which sequences is $\eff(S_{n,k},\bbzn)\sim1/\log k$?
\end{question}

Our next corollary shows that `typical' $k$-element subsets of $\bbZ$ cover
efficiently. To make sense of this, we choose our subsets from a large
interval $[n]=\{1,2,\ldots,n\}$, and take limits as $n\to\infty$.

\begin{corollary}
  \label{CJ2}
Let $S_{n,k}'$ denote a $k$-element subset of $[n]$ chosen uniformly at random.
For every fixed $k$ we have $\eff(S_{n,k}',\bbZ)\pto1$ as \ntoo.
\end{corollary}

\begin{proof}
Using the obvious coupling between random subsets
of $[n]$ and random subsets of $\bbZ_n$, \eqref{l2b} gives
$\eff(S_{n,k}',\bbZ)\ge\eff(S_{n,k},\bbzn)$. The result follows by \refC{CJ1}.
\end{proof}

One can also use the results above to show that certain deterministic
sequences of sets $S$ are asymptotically efficient.
\begin{example}
Fix distinct integers $a$ and $b$, and let $S=S_n\=\set{a,b,n}$.
It is easily seen that, for any given $C$, if $n$ is large enough
then the conclusion of \refT{TJ1Z} cannot hold for this $C$.
Thus \refT{TJ1Z} implies that
$\eff(S_n,\bbZ)\to1$ as \ntoo.
(See also \cite{Schmidt,Tuller}.)
\end{example}

In the appendix
we give a new proof of Schmidt's results; it seems to us that
our proof is different from that in~\cite{Schmidt}.
Ours is based on the following
number theoretic lemma, which may be seen
as a discrete version of Kronecker's Theorem (see, for
example,~\cite[Chapter XXIII]{HardyWright}), which states
that if $x_1,\dots,x_m$ are real numbers that are
linearly independent over $\bbQ$,
then the set of multiples of $(x_1,\dots,x_m)$
is dense in the torus $\bbR^m/\bbZ^m$.
It may well be that this lemma is known, but we have not 
been able to find a reference, perhaps because much more attention
has been paid to quantitative versions of Kronecker's Theorem
than to discrete ones. As usual, we write $(x,y)$ for the greatest common divisor of
two integers $x$ and $y$.

\begin{lemma}
  \label{LJ1strong}
For every $m\ge1$ and $\gd>0$ there exists $C=C(m,\gd)$ such that
for every $n$ and every $b_1,\dots,b_m\in\bbZ$,
if $\q x_1,\dots,\q x_m$ are any $m$ elements of\/
$\bbzn$, then at least one of the following holds.
\begin{romenumerate}
  \item
There exists a vanishing linear combination
$\sum_{i=1}^m a_i\qx_i =0$ in $\bbzn$ with
integer coefficients $a_1,\dots,a_m$
satisfying $0<\max_i|a_i|\le C$.
\item
There exist $z\in\bbZ$ with $(z,n)=1$ and $y_1,\dots,y_m\in [0,\gd n]\cap\bbZ$ such
that $z\qx_i-\q b_i=\q y_i$ in $\bbzn$, \ie,
\begin{equation*}
  z x_i-b_i\equiv y_i \pmod n,
\qquad
i=1,\dots,m.
\end{equation*}
\end{romenumerate}
\end{lemma}

For the proof, see the appendix.

\section{$G=\bbR$}\label{SR}

In this section we consider coverings of the other basic non-compact group, 
namely $\bbR$.
Of course, if the union $T+S$ of a set of translates $t+S$ of a given set $S$ covers
$\bbR$, then $S$ and $T$ cannot both be discrete. There are two natural analogues
of the results and problems discussed so far in this paper. In one,
$S$ is discrete (typically finite), and the `size' or `density' of $T$
is measured using the Lebesgue measure. Such coverings
are considered implicitly by, for example,
Rohlin~\cite{Rohlin} and Laczkovich~\cite{Laczkovich},
and explicitly by Schmidt~\cite{Schmidt}.
It seems closer to the spirit of the original papers of
Erd\H os~\cite{Erdos}, Lorentz~\cite{Lorentz} and Newman~\cite{Newman_dens},
however, to consider cases where $T$ is discrete. Then $S$ cannot be; usually,
we will take $S$ to be some `nice' set, such as a finite union of intervals.

We shall only consider coverings by 
measurable sets $S\subset\bbR$ with non-empty interior.
Typically, $\nn(S,\bbR)=\infty$, so we need to modify the definitions in \refS{Scompact}.
Since the modifications from the compact case to the case $G=\bbR$ 
are closely analogous to those from the finite case to the case $G=\bbZ$,
we describe these changes only briefly.

Let $\gl$ denote the Lebesgue measure.
For $x>0$, let $\nn(S,x)$
be the smallest number of translates of $S$ that cover
the interval $[0,x]$, \ie,
\begin{equation}\label{trx}
  \nn(S,x)\=\min\set{|T|:T+S\supseteq [0,x]}.
\end{equation}
Obviously, the number of translates required to cover any other
interval of length $x$ is the same, and
\begin{equation}
\nn(S,x+y)\le\nn(S,x)+\nn(S,y)  
\end{equation}
for all $x,y>0$. Hence $\nn(S,x)/x$ converges as $x\to\infty$,
and the limit satisfies
\begin{equation}\label{tr}
  \nn(S)\=\lim_\xtoo \nn(S,x)/x = \inf_{x>0} \nn(S,x)/x. 
\end{equation}
Assuming, as we do, that $S$ has non-empty interior, then
this \emph{covering density} $\nn(S)$ is finite.

Using \eqref{tr} in place of \eqref{tz} and \eqref{tzinf},
and working with $\gl(S)$ in place of $|S|$,
we may adapt 
the definitions of $\nn(S)$, $\kk(S)$ and $\eff(S)$ as in \refS{SZ}.
We may use the notation $\nn(S,\bbR)$ \etc{} for emphasis and
clarity.
Similarly, we can make sense of the efficiency of a particular covering,
at least when it is periodic.

In Sections \ref{SRgreedy} and \ref{SR2}, 
we shall consider the case when $S$ is a
union of $k$ closed intervals $I_i$: $S=\bigcup_{i=1}^k I_i$.
The case $k=1$ is trivial; then
$\nn(S,x)=\ceil{x/\gl(S)}$, so $\nn(S)=\gl(S)\qw$ and
$\kk(S)=\eff(S)=1$.

For larger $k$ we obtain a trivial bound by ignoring all but one of
the intervals, $I_j$, say, and using the cover
$\cT\=\set{ia_j+S:i\in\bbZ}$ where $a_j\=\gl(I_j)$. 
This covering is periodic, and has density
$\nn(\cT)=a_j\qw$; hence its efficiency is
$\eff(\cT)=a_j/\gl(S)=a_j/\sum_i a_i$.
Choosing $j$ so that $a_j$ is maximal, this yields
\begin{equation}\label{eai}
  \eff(S)\ge\frac{\max_i a_i}{\sum_i a_i}.
\end{equation}
In particular, for any set $S$ that is a union of $k$ (closed)
intervals, 
\begin{equation}\label{es1k}
  \eff(S)\ge\frac1k.
\end{equation}

We define 
\begin{align}
\gb_k&\=
\inf\Bigset{\eff\Bigpar{\bigcup_1^k I_i}:
I_i \text{ intervals in $\bbR$}},
\label{beta}
\\
\gam_k&\=
\inf\Bigset{\eff\Bigpar{\bigcup_1^k I_i}:
I_i \text{ intervals in $\bbR$ of the same length}}.
\label{gamma}
\end{align}
Obviously, by making small magnifications of the intervals, we may
restrict the intervals in \eqref{beta} and \eqref{gamma}
to be closed (or open, or half-open) without changing the infima.
Also, we may require the intervals in \eqref{gamma} to all have
unit length.

From \eqref{es1k} and trivial inequalities
we have $k\qw \le \gb_k\le\gam_k \le 1$. The lower bound will be improved in
\refS{SRgreedy}. 
It is interesting to try to find $\gb_k$ and $\gam_k$ exactly for
small $k$. 
Trivially, $\gb_1=\gam_1=1$; we shall prove that $\gb_2=\frac23$ 
and $\gam_2=\frac34$
in \refS{SR2}.

\begin{theorem}\label{th_abc}
For every $k\ge1$,
$\gb_k\le\gam_k\le\ga_k\le1$.
\end{theorem}
\begin{proof}
We have already observed the first inequality.

For the second, it is slightly simpler to use half-open (or open)
intervals.
Given $S=\set{s_1,\dots,s_k}\subset\bbZ$, let
$I_j=[s_j,s_j+1)$, and set $\tS\=\bigcup_{j=1}^k I_j = S+[0,1)$.

Suppose that $T=\set{t_1,\dots,t_l}\subset\bbR$ is such that
$T+\tS\supseteq[0,n]$.
If $m\in\{1,2,\ldots,n\}$, then $m\in T+\tS$,
so $m\in t_i+I_j$ for some $t_i\in T$ and $1\le j\le k$.
In other words, $m=t_i+s_j+x$ for some $t_i\in T$, $s_j\in S$ and
$x\in[0,1)$. Hence $t_i+x=m-s_j\in\bbZ$, so $t_i+x=\ceil{t_i}$ and
$m=\ceil{t_i}+s_j$. Consequently, 
setting $\hT\=\set{\ceil{t}:t\in T}$, we have $\hT+S\supseteq[n]$.

This shows that the number $\nn(S,n;\bbZ)$ of translates of $S$
needed to cover $\{1,2,\ldots,n\}$ is at most 
$|\hT|\le|T|$. Minimizing over $T$, we thus have 
$\nn(S,n;\bbZ)\le\nn(\tS,n;\bbR)$, where $\nn(\tS,n;\bbR)$,
as in \eqref{trx},
is the minimum number of translates of $\tS\subset \bbR$ needed
to cover $[0,n]\subset\bbR$.
From \eqref{tzinf} and \eqref{tr}, it follows that
\begin{align*}
 \nn(S,\bbZ)&\le\nn(\tS,\bbR).
\intertext{Since $\gl(\tS)=|S|$, the definitions \eqref{kz} and \eqref{ez}
and their analogues for $\bbR$ then give}
\kk(S,\bbZ)&\le\kk(\tS,\bbR),
\\
\eff(S,\bbZ)&\ge\eff(\tS,\bbR)\ge\gam_k.
\end{align*}
Taking the infimum in \eqref{ga} we find $\ga_k\ge\gam_k$.
\end{proof}
Note
that since $\alpha_2=1$, \refC{CR2} in \refS{SR2} shows that
the inequalities $\gb_k\le\gam_k\le\ga_k$ are strict for
$k=2$; we conjecture that these inequalities are strict for all larger
$k$ too.

\section{Improved bounds for unions of intervals in $\bbR$}\label{SRgreedy}

It seems difficult to use the greedy algorithm directly on $\bbR$,
since it is not obvious when a small residual uncovered set really is
empty.
But we can use our result for $\bbZ$, which was based on the greedy
algorithm in $\bbZ_n$ for large $n$.

\begin{theorem}\label{TRgreedy}
  If $S=\bigcup_{i=1}^k I_i$ is a union of $k\ge 2$ intervals in $\bbR$, then
  \begin{equation*}
\kk(S)\le \log k + \log\log k + 5,
  \end{equation*}
and thus $\eff(S)\ge(1-o(1))/\log k$.
\end{theorem}

\begin{proof}
Recall that $\gl$ denotes the Lebesgue measure.
By homogeneity, we may assume that $\gl(S)=1$.
Fix $\gd>0$ 
and consider the regularly spaced intervals
$J_j\=[j\gd,(j+1)\gd]$ of length $\gd$.
Let $U\=\set{j\in\bbZ:J_j\subseteq S}$ and set
$S'\=\bigcup_{j\in U}J_j\subseteq S$,
so $S'$ is the maximal subset of $S$ consisting of closed intervals
all of whose endpoints are multiples of $\delta$.

Assuming, as we may, that $I_1,\dots,I_k$ are disjoint, we have
$U=\bigcup_{i=1}^k U_i$, where $U_i=\set{j\in\bbZ:J_j\subseteq I_i}$.
Since $I_i\setminus\bigcup_{j\in U_i} J_j$ consists of at most two
intervals (at each end of $I_i$) of lengths $<\gd$, we have
$|U_i|\gd>\gl(I_i)-2\gd$ and thus
\begin{equation}
  \label{jb}
\gl(S')=|U|\gd=\sumik|U_i|\gd
>\sumik\gl(I_i)-2k\gd=\gl(S)-2k\gd
=1-2k\gd.
\end{equation}
We will choose $\gd<1/(2k)$, so $S'\neq\emptyset$.

Since the intervals $J_j$ have equal lengths and form a partition of
$\bbR$ (except for common endpoints), it
is obvious that 
$\nn(S',n\gd)\le\nn(U,n)$. Taking limits
(recalling \eqref{tz} and \eqref{tr}),
it follows that $\nn(S',\bbR)\le\nn(U,\bbZ)/\gd$, so
$\kk(S',\bbR)\le\kk(U,\bbZ)$.
Consequently, \refT{TZgreedy} yields
\begin{equation}
  \kk(S',\bbR)\le\kk(U,\bbZ)\le\log|U|+1
=\log\frac{\gl(S')}{\gd}+1.
\end{equation}
Since $S'\subseteq S$, we have $\nn(S,\bbR)\le\nn(S',\bbR)$
and thus, using $\gl(S')\le\gl(S)=1$ and \eqref{jb},
\begin{equation}
  \kk(S,\bbR)\le\frac{\gl(S)}{\gl(S')}\kk(S',\bbR)
\le
\frac1{\gl(S')}\Bigpar{\log\frac{\gl(S')}{\gd}+1}
\le
\frac1{1-2k\gd}\Bigpar{\log\frac{1}{\gd}+1}.
\end{equation}
Choosing $\gd=\xpar{2k(\log k+1)}\qw$ (which is close to optimal), this yields
\begin{align}
  \kk(S,\bbR)
&\le
\frac{\log k+1}{\log k}\bigpar{\log 2+\log k+\log(\log k+1)+1}
\label{rk1}
\\
&\le
\Bigpar{1+\frac{1}{\log k}}
\Bigpar{\log k+\log\log k+1+\log2+\frac1{\log k}}.
\label{rk2}
\end{align}
This bound is evidently $\log k+\log\log k+O(1)$.
To be more precise, it is easily verified that the difference between the 
final bound in
\eqref{rk2} and $(\log k+\log\log k)$ is decreasing, and a numerical
calculation verifies that this difference is smaller than $5$ when $k=6$, which
completes the proof for $k\ge6$. The trivial bound $\kk(S)\le k$
is enough when $k\le7$.
\end{proof}

Together, Theorems \ref{th_abc} and \ref{TRgreedy} show that
$(1-o(1))/\log k\le \gb_k\le\gam_k\le\ga_k\le1$.
By \refT{Ta}, the asymptotics are the same for all three quantities:
$\gb_k\sim\gam_k\sim\ga_k\sim 1/\log k$ as $k\to\infty$.

\section{Two intervals in $\bbR$}\label{SR2}

Let us consider the case $k=2$, when $S=I_1\cup I_2$ is a union of two
closed intervals. We may assume that the intervals are disjoint, since
otherwise $S$ is an interval and thus $\eff(S)=1$.

\begin{theorem}
  \label{TR2}
  \begin{thmenumerate}
\item
If $S=I_1\cup I_2$ where $I_1$ and $I_2$ are two closed intervals,
then $\eff(S)> 2/3$.
\item
If in addition $I_1$ and $I_2$ have the same length, then
$\eff(S)> 3/4$.
  \end{thmenumerate}
\end{theorem}

\begin{proof}
We shall prove the inequalities by giving explicit constructions of
coverings with these efficiencies. As we do not have
a single covering algorithm, we shall consider
different constructions for different cases.

Let $a\=\gl(I_1)$ and $b\=\gl(I_2)$ be the lengths of the
intervals, and let $c>0$ be the gap between them. By symmetry we may
assume that $a\ge b>0$ and that $I_1$ lies to the left of $I_2$; we
also assume that the left endpoint of $I_1$ is 0. Thus, $I_1=[0,a]$ and
$I_2=[a+c,a+c+b]$. 

\mm{I}
Ignore $I_2$ and use only $I_1$ to cover by $\cT=\set{S+na:n\in\bbZ}$.
Clearly, $\nn(\cT)=\nn(I_1)=a\qw$ and thus $\kk(S)\le\kk(\cT)=(a+b)/a$
and,
as in \eqref{eai},
\begin{equation}
\eff(S)\ge  \eff(\cT)=\frac1{\nn(\cT)\gl(S)} 
= \frac{a}{a+b}.
\end{equation}

\mm{II}
Assume $c\le b$. Then $S\cup(S+b)=[0,a+c+b+b]$
is an interval of length $a+2b+c$. 
We may thus cover $\bbR$ by
$\cT\=\set{S+ib+j(a+2b+c):i=0,1,\;j\in\bbZ}$.
This covering is periodic, with period $a+2b+c$, and 2 translates of
$S$ begin in each period. Hence
$\nn(\cT)=2/(a+2b+c)$, and the
efficiency is
\begin{equation}
\eff(S)\ge  \eff(\cT)=\frac1{\nn(\cT)\gl(S)} 
= \frac{a+2b+c}{2(a+b)},
\qquad c\le b.
\end{equation}

\mm{III}
Consider first the translates
\set{S+i(a+b+c):i\in\bbZ}.
Since the right endpoint of $S$ equals the left endpoint of
$S+(a+b+c)$, these translates match up perfectly, and the union
$S'\=\bigcup_{i\in\bbZ}(S+i(a+b+c))$ consists of an infinite sequence
of intervals of length $a+b$ with gaps of length $c$ in between.

Let $y\=c/(a+b)$ and let $m\=\ceil y$. Then
$\bigcup_{j=0}^m(S'+j(a+b))=\bbR$, and thus 
$\cT\=\set{S+i(a+b+c)+j(a+b):i\in\bbZ,\,j=0,\dots,m}$ is a covering of
$\bbR$.
This covering is periodic, with period $a+b+c$, and $m+1$ translates of
$S$ begin in each period. Hence $\nn(\cT)=(m+1)/(a+b+c)$ and the
efficiency is
\begin{equation}
\eff(S)\ge  \eff(\cT)=\frac1{\nn(\cT)\gl(S)} 
= \frac{a+b+c}{(m+1)(a+b)}
=\frac{y+1}{m+1}
=\frac{y+1}{\ceil{y+1}}.
\end{equation}

To complete the proof of the first part,
we shall show for any positive $a,b,c$ with $a\ge b$, at least one of
these methods yields $\eff(S)\ge\eff(\cT)>2/3$.
Firstly, if $b<a/2$, then  \mmx I will do. Further, $x/\ceil x>2/3$
for all $x>4/3$, so \mmx{III} will do when $y>1/3$.
This leaves only the case $a\le 2b$ and $y\le1/3$. In this case
$c=y(a+b)\le(a+b)/3\le b$, so \mmx{II} applies, and yields
\begin{equation*}
\eff(S)
\ge
\frac{a+2b+c}{2(a+b)}  
>
\frac{a+2b}{2a+2b}
=1-\frac{a}{2a+2b}
\ge2/3,
\end{equation*}
using again $2b\ge a$. This proves that $\eff(S)>2/3$ in all cases,
which proves (i).

For (ii), we  specialize to $a=b$. We shall need yet another covering algorithm.

\mm{IV}
Let $z\=c/a$ and let $m\=\ceil z\ge1$. 
Consider first the translates \set{S+ia:i=0,\dots,m}.
It is easy to see that $S'\=\bigcup_{i=0}^m(S+ia)$ is an interval of
length $a+c+(m+1)a=(m+2)a+c$. 
We may thus cover $\bbR$ by
$\cT\=\set{S+ia+j((m+2)a+c):i=0,\dots,m,\;j\in\bbZ}$.
This covering is periodic, with period $(m+2)a+c$, and $m+1$ translates of
$S$ begin in each period. Hence $\nn(\cT)=(m+1)/((m+2)a+c)$ and the
efficiency is
\begin{equation*}
  \begin{split}
\eff(S)&\ge  \eff(\cT)=\frac1{\nn(\cT)\gl(S)} 
= \frac{(m+2)a+c}{2(m+1)a}
=\frac{m+2+z}{2m+2}
=\frac{m+2+z}{\ceil{m+2+z}}
\\&
>3/4,	
  \end{split}
\end{equation*}
since $m\ge1$ and $x/\ceil x>3/4$ for all $x>3$.
\end{proof}

We next give two examples showing that the bounds
given in \refT{TR2} are indeed best possible.

\begin{example}\label{ER1}
Let $S=[0,2]\cup[3+\weps,4+\weps]$ for $\weps>0$; thus $\gl(S)=3$.
Obviously, we can ignore the shorter interval in $S$ and tile
$\bbR$ by translates of the larger one; this shows
that $\nn(S)\le 1/2$ and so $\eff(S)=1/(3\nn(S))\ge 2/3$.
\refT{TR2} shows that we can do a little better
(for example using method (iii) in the proof above), i.e.,
that $\eff(S)>2/3$. However, as we shall now see, for $\weps$
small we cannot do much better: the shorter interval is essentially useless.

Suppose that $\{t_i+S\}_{i=1}^N$ is a covering of a (long) interval $[0,x]$.
We may assume that $t_1<t_2<\cdots<t_N$. Let $k_0=\inf\{k:t_k\ge 0\}$,
and let $k_1=\sup\{k:t_k < x-(4+\weps)\}$.

Suppose that $k_0\le k\le k_1$. Then the
`gap' $I=(t_k+2,t_k+3+\weps)$ in $t_k+S$ lies within $[0,x]$, and is
thus covered by the sets $t_i+S$.

We consider three cases. Suppose first that $I$ is covered entirely
by `later' translates, i.e., sets of the form $t_i+S$ with $i>k$.
Then the left endpoint of $I$ lies in the first interval $t_i+[0,2]$ of
some such $t_i+S$, so we have 
\begin{equation}\label{case1}
 t_{k+1}\le t_i \le t_k+2.
\end{equation}
Suppose next that $I$ is covered partly by some $t_i+S$, $i<k$,
and partly by some $t_j+S$, $j>k$. Then two such sets must meet,
and
\begin{equation}\label{case2}
 t_{k+1} - t_{k-1} \le t_j-t_i  \le 4+\weps.
\end{equation}
Finally, suppose that $I$ is covered entirely by sets $t_i+S$, $i<k$.
In any such set, only the shorter interval can meet $I$, so it takes
at least two such sets $t_i+S$, $t_j+S$ to cover $I$, and assuming
that $i<j$ we have
\begin{equation*}
 t_{k-2}\ge t_i \ge (t_k+2)-(4+\weps),
\end{equation*}
so $t_k\le t_{k-2}+2+\weps$.
But we always have $t_{k+1}\le t_k+\sup S-\inf S=t_k+4+\weps$, so in this case
\begin{equation}\label{case3}
 t_{k+1} \le t_{k-2} +6+2\weps.
\end{equation}

Combining \eqref{case1}, \eqref{case2} and \eqref{case3},
we see that in all cases there is some $1\le i\le 3$ such that
$t_{k+1}\le t_{k+1-i} + i(2+\weps)$, 
and it follows by induction that $t_k\le (2+\weps)(k-k_0+2)$ for all $k\in [k_0-2,k_1+1]$.
Since $t_{k_1+1}\ge x-(4+\weps)=x-O(1)$,
it follows that
the number $N$ of translates satisfies
$N\ge k_1-k_0+1\ge x/(2+\weps)-O(1)$.
Since $x$ was arbitrary we have $\nn(S)\ge 1/(2+\weps)$
and $\eff(S)=1/(\gl(S)\nn(S))\le (2+\weps)/3$.
\end{example}

The next example shows that for a union $S$ of two intervals of equal
lengths, the bound $\eff(S)>3/4$ given by \refT{TR2} is best possible.

\begin{example}\label{ER2}
Let $S=[0,1]\cup [1+\weps,2+\weps]$ where $\weps>0$.
Suppose that $S_1,\ldots,S_N$ are translates of $S$ covering $I=[0,x]$,
and hence covering $I'=[2+\weps,x-(2+\weps)]$.
For $y\in I'$ let $f(y)$ be the number of times $y$ is covered, so
$f(y)=\sum_i \chi_{S_i}(y)$, where $\chi_A$ denotes the characteristic
function of a set $A$.

Set
\begin{equation*}
 w(S_i) = \int_{S_i\cap I'} \frac{\dd y}{f(y)} = \int_{I'} \frac{\chi_{S_i}(y)}{f(y)}\dd y.
\end{equation*}
Note that $w(S_i)$ may be thought of as the amount of `covering work' done by $S_i$,
where the work in covering a point multiple times is split between the sets covering it.
Then
\begin{equation}\label{wtot}
 \sum_{i=1}^N w(S_i) = \int_{I'} \frac{\sum_i \chi_{S_i}(y)}{f(y)} \dd y = \int_{I'} 1\dd y = x-4-2\weps.
\end{equation}

Suppose that $w(S_i)>0$, so $S_i$ meets $I'$ and hence $S_i$ lies entirely within $[0,x]$.
Let $I_i$ be the gap of length $\weps$ in $S_i$. Then $I_i$ is covered by $\bigcup_{j\ne i} S_j$.
In particular, some $S_j$ meets $I_i$. But then $S_j$ contains a unit interval $J$
meeting $I_i$. Then $J\subset S_i\cup I_i$, so
$\gl(S_j\cap S_i)\ge \gl(J\cap S_i)=1-\gl(J\cap I_i)\ge 1-\weps$.
Since $f(y)\ge 2$ on $S_j\cap S_i\cap I'$, 
we have $\chi_{I'}(y)/f(y)\le 1/2$ on a subset of $S_i$ of measure at least $1-\weps$.
Since $\chi_{I'}(y)/f(y)\le 1$ on all of $S_i$, 
it follows that
\begin{equation*}
 w(S_i) = \int_{S_i} \frac{\chi_{I'}(y)}{f(y)}\dd y \le \gl(S_i)-(1-\weps)/2 =(3+\weps)/2.
\end{equation*}
Since $S_i$ was any translate with $w(S_i)>0$,
referring back to \eqref{wtot} we see that
$N(3+\weps)/2 \ge (x-4-2\weps)=x-O(1)$.
Since $x$ was arbitrary, it follows that $\nn(S)\ge 2/(3+\weps)$,
so $\eff(S)=1/(\nn(S)\gl(S))\le (3+\weps)/4$.
\end{example}

Together, \refT{TR2} and Examples \refand{ER1}{ER2} establish the result
we promised earlier.
\begin{corollary}\label{CR2}
For $k=2$ the quantities defined in \eqref{beta} and \eqref{gamma} satisfy
$\gb_2=\frac23$ and $\gam_2=\frac34$.
\nopf
\end{corollary}

\section{$G=\bbZ^d$}\label{SZd}

In this and the next section we briefly consider coverings of higher dimensional
spaces, starting with $\bbZ^d$. In dimensions two and higher the notion 
of the density of a general set $T\subset \bbZ^d$ is a little slippery:
it is naturally defined as the limit (if it exists) of the fraction of a large
`ball' that lies in $T$, but the existence and value of the limit
may depend on the norm chosen to define the balls. This ambiguity
does not arise for \emph{periodic} sets $T$, i.e., sets invariant
under translation through the elements of some lattice $\cL\subset\bbZ^d$.
If $S$ is finite and $T+S=\bbZ^d$, then for any reasonable notion of density
there are periodic sets $T'$ with density arbitrarily close to that of $T$
such that $T'+S=\bbZ^d$, so the covering density of a finite set $S$ turns
out to make very good sense. Rather than discuss this further, we simply
pick one explicit definition.

Given a finite, non-empty subset $S$ of $\bbZ^d$, $d\ge 2$, let
$\nn(S,[n]^d)$ be the smallest number of translates of $S$ that cover
the cube $[n]^d$.
For $n$ large (greater than the maximum difference between corresponding
coordinates of points in $S$) we may regard $S$ as a subset
of the discrete torus $\bbZ_n^d$. Then \eqref{l1} extends to
$\nn(S,\bbZ_n^d)\le \nn(S,[n]^d)$. Moreover, the proof of \refL{LZn} is easily
modified to give the following result.

\begin{lemma}
If $S\subset \bbZ^d$ is finite and non-empty, then the limits
\begin{equation*}
 \lim_{n\to\infty} \frac{\nn(S,[n]^d)}{n^d} \qquad\hbox{and}\qquad
 \lim_{n\to\infty} \frac{\nn(S,\bbZ_n^d)}{n^d}
\end{equation*}
exist and are equal.
\nopf
\end{lemma}
The \emph{covering density} $\nn(S)$ of $S$ is defined to be
the common value of the limits above and, as in \refS{SZ}, the \emph{efficiency}
of $S$ is
\begin{equation}\label{pz2}
 \eff(S)\=1/(\nn(S)|S|)=\lim_{n\to\infty}\eff(S,\bbZ_n^d).
\end{equation}

By \eqref{g2f}, if $|S|=k$ then $\nn(S)\le H_k/k$ and $\eff(S)\ge 1/H_k$,
just as for subsets of $\bbZ$.

One interesting case is when $S$ is a product set: suppose that $S=S_1\times S_2$
with $S_1\subset \bbZ^{d_1}$ and $S_2\subset \bbZ^{d_2}$. Then \eqref{pz2}
and \refL{Lprod} show that
\begin{equation}\label{pz3}
 \eff(S_1)\eff(S_2) \le \eff(S_1\times S_2)\le \min\{\eff(S_1),\eff(S_2)\}.
\end{equation}
The argument in \refE{Eprod} and \refT{Ta} show that the first inequality may be strict.

Turning to sets of a given size, let $\ga_{k,d}=\inf\{\eff(S): S\subset\bbZ^d,\,|S|=k\}$.
Taking $S_2=\{0\}$ in \eqref{pz3}, it follows that
$\ga_{k,1}\ge \ga_{k,2}\ge \cdots$.

\begin{question}\label{dim}
Is $\ga_{k,d}$ equal to $\ga_k$ ($=\ga_{k,1}$) for $d>1$?
\end{question}

Note that the lower bound $\ga_{k,d}\ge 1/H_k$ holds for every $d$, as a consequence of 
\eqref{g2fe} and \eqref{pz2}.

In general, subsets of $\bbZ^d$ seem much harder to handle than subsets of $\bbZ$.
For example, unlike in the case $d=1$ discussed in \refR{Ralg},
we do not know of an algorithm that given as input a finite set $S\subset\bbZ^d$
calculates $\eff(S)$;
indeed, it may well be the case that no such algorithm exists. If we
consider a finite set $S_1,\ldots,S_k$ of finite subsets of $\bbZ^d$,
and simply ask whether there is a partition of $\bbZ^d$ into sets each of which
is a translate of some $S_i$, then this a form of the classical
domino tiling problem discussed by Wang~\cite{Wang}. Usually, one considers
square tiles with coloured edges, and asks whether $\bbZ^2$ can be tiled by translates
of a given set of such \emph{Wang tiles} so that the colours on adjacent edges match;
it is easy to code any problem of this form with suitable sets $S_i$.
As shown by Berger~\cite{Berger},
this tiling problem is undecidable; a key 
related fact is the existence of sets of tiles which do tile $\bbZ^2$,
but not in a periodic way.

Returning to our problem, it may be
that the question of deciding whether a given $S$ has a covering
of a certain efficiency is undecidable. Moreover, it seems
to be still open to determine whether there is always
an optimal periodic covering, or indeed whether $\eff(S)$ is necessarily rational.
In fact, even for the special case of tilings,
i.e., sets with $\eff(S)=1$, the corresponding questions
seem to be still open; see, for example,
Rao and Xue~\cite{RaoXue}.
There are some partial results: for example,
Wijshoff and van Leeuwen~\cite{WL}
showed that if $S\subset\bbZ^2$ is connected and contains no `holes' (so $S$
corresponds to a polyomino), then $S$ tiles $\bbZ^2$ if and only if it
does so in a periodic manner, and gave an algorithm to determine whether such 
a tiling exists (see also Beauquier and Nivat~\cite{BN}). In a different
direction, Schmidt and Tuller~\cite{SchmidtTullerII} give a criterion
for deciding whether a set $S$ with $|S|=4$ tiles $\bbZ^d$.


\section{$G=\bbR^d$}\label{SRd}

Finally, one can also consider subsets $S$ of $\bbR^d$, in the same 
way as in \refS{SR}. Here, many classes of subsets seem natural to consider.
For example, as noted in the introduction,
the case when $S$ is a ball is a classical problem, and the case where $S\subset\bbR^2$ is convex
polygon has also received much attention; see~\cite{FT53,FT72,Rogers}.
As in the previous section, there is an apparent difficulty defining
the covering density of $S$, but for the bounded sets $S$ we consider,
one need only consider periodic coverings, and the problem disappears;
we omit the details.

Another interesting case, generalizing that discussed in \refS{SRgreedy}, is when 
$S\subset\bbR^2$ is a union of $k$  rectangles $R_i$ of the
form $R_i=I_i\times J_i$, which we may or may not insist are disjoint.
(Of course, this extends to $d\ge 3$.)
The discretization argument used in the proof of \refT{TRgreedy} does not work
in this context: even if they have the same area, the rectangles
may have very different aspect ratios, and there may be no `grid'
that approximates all of them well. In fact, we do
not know any lower bound for the efficiency $\eff(S)$ better than the trivial bound $\eff(S)\ge 1/k$,
obtained by ignoring all rectangles in $S$ except one with maximal area.

\begin{question}\label{Q1/k}
Writing $\gb_{k,d}$ for the infimum of $\eff(S)$ over all sets $S\subset \bbR^d$
that are unions of $k$ axis-aligned rectangles, do we have $\gb_{k,d}=1/k$
for all (or any) $k\ge 2$ and $d\ge 2$?
\end{question}

In dimension two, one natural candidate for a very inefficient set $S$ with $k=2$ is the `thin cross' 
$S_L=([-L,L]\times [-1,1]) \cup ([-1,1] \times [-L,L])$ with $L$ large.
However, as noted by Everett and Hickerson~\cite{EH} (whose main purpose
in that paper was to relate suitable coverings of $\bbR^d$ to corresponding ones of $\bbZ^d$), 
$\eff(S_L)\ge 3/4+o(1)$ as $L\to\infty$.
More precisely, they conjectured that when $L=2a+1$ with $a$ an integer
(so $S_L$ is the union of $4a+1$ congruent squares), then
$\kk(S_L)=(4a+1)/(3a+2)$, and proved that $\kk(S_L)$ is at most this large
by giving an example of a covering. Loomis~\cite{Loomis} proved their
conjecture for $a=2$, but in
general it seems to be still open. For a survey of related questions
concerning crosses and `semi-crosses' in various dimensions, see Stein~\cite{Stein_survey}.

Perhaps a better candidate for the union $S$ of two rectangles with $\eff(S)$ small
is the set $S_L'=([1,L]\times [0,1]) \cup ([0,1]\times [1,L])$
formed by two arms of a cross without the central square.
In this case, for large $L$
we do not have a better lower bound on $\eff(S_L')$ than the trivial $\eff(S_L')\ge 1/2$.

We are so far from knowing the answer to \refQ{Q1/k} above that we do not even 
know whether $k\gb_{k,2}$ is bounded as $k\to\infty$. 
One plausible approach to proving that it is bounded is to consider
sets $S$ formed as the union of $k$ rectangles with a common centre
and unit area,
but with very different aspect ratios. But even in this case
we do not know the covering efficiency; it may be that even if the
aspect ratios form a rapidly increasing sequence, $\eff(S)$ is larger than $1/k$
by more than a constant factor as $k\to\infty$.
If so, then, fitting $S$ inside the region $H$ bounded by the hyperbola $|x||y|=1$, 
it would follow that for any $\eps>0$ there is a \emph{periodic} covering $T+H$ of $\bbR^2$
by hyperbolic regions in which $T$ has (upper) density at most $\eps$.
(It is not hard to construct an unrestricted covering in which the density
of $T$, defined as the limit of $(\pi r^2)^{-1}$ times the number of points within a disk of radius $r$,
is zero.)
The question
of whether such periodic coverings exists seems of interest in its own right;
it has the flavour of number theory/approximation
theory, and may well be known, but we have not found a reference.

Of course, one can also consider many other families of sets,
such as unions of rectangles with arbitrary orientations. In general,
it seems that most questions one can ask in dimension two or higher are
rather difficult.

\appendix
\section{}

In this appendix we prove Lemma~\ref{LJ1strong}, and show that Theorem \ref{TJ1} 
follows. Our proof of \refT{TJ1} is simpler 
in the case when $n$ is prime, so
we begin with this case (which suffices to prove \refC{CJ2}) and return to 
the general case later.

To handle the case where $n$ is prime, the following weaker form of 
Lemma \ref{LJ1strong}
suffices; the only difference is that we have omitted the condition $(z,n)=1$
from part (ii).
Recall that $\q x$ denotes the element in $\bbzn$ corresponding to an
integer $x\in\bbZ$.

\begin{lemma}
  \label{LJ1weak}
For every $m\ge1$ and $\gd>0$ there exists $C=C(m,\gd)$ such that
for every $n$ and every $b_1,\dots,b_m\in\bbZ$,
if $\q x_1,\dots,\q x_m$ are any $m$ elements of\/
$\bbzn$, then at least one of the following holds.
\begin{romenumerate}
  \item
There exists a vanishing linear combination
$\sum_{i=1}^m a_i\qx_i =0$ in $\bbzn$ with
integer coefficients $a_1,\dots,a_m$
satisfying $0<\max_i|a_i|\le C$.
\item
There exist $z\in\bbZ$ and $y_1,\dots,y_m\in [0,\gd n]\cap\bbZ$ such
that $z\qx_i-\q b_i=\q y_i$ in $\bbzn$, \ie,
\begin{equation*}
  z x_i-b_i\equiv y_i \pmod n,
\qquad
i=1,\dots,m.
\end{equation*}
\end{romenumerate}
\end{lemma}

\begin{proof}
We use Fourier analysis. Set $\go:=\exp(2\pi\ii/n)$, and
define the Fourier transform of a function $f:\bbzn\to\bbC$ by 
\begin{equation*}
  \ff f(\q a)=\sum_{\qx\in\bbzn}f(\qx)\go^{ax},
\qquad a\in\bbZ. 
\end{equation*}

If $\gd\ge 1$ then (ii) holds trivially for any $z$.
Provided we choose $C\ge 4/\gd$, as we may,
if $n\le 4/\gd$ then (i) holds trivially with all $a_i$ equal to $n$.
Hence we may assume that $\gd<1$ and $n>4/\gd$.

Let $\ell\=\floor{\gd n/2}$ and $I\=\set{0,\dots,\ell}\subset\bbzn$.
Define $g:\bbzn\to[0,\infty)$ by $g:=n\qww\chi_I*\chi_I$, so
\begin{equation*}
 g(\qx)=n^{-2}\sum_{\q y\in\bbzn}\chi_I(\q x-\q y)\chi_I(\q y).
\end{equation*}
For $i=1,\dots,m$, define $h_i$ by setting
$h_i(\qx)=g(\qx-\q b_i)$.
Note that
\begin{equation}
  \label{j1}
\bigset{x\in\set{0,\dots,n-1}:g(\qx)\neq0}\subseteq[0,2\ell]\subseteq[0,\gd n].
\end{equation}

For every $a\in\bbZ$,
\begin{equation*}
 \bigabs{\ff{\chi_I}(\q a)}
=
\lrabs{\sum_{x=0}^{\ell}\go^{ax}}
=\lrabs{\frac{1-\go^{a(\ell+1)}}{1-\go^a}}
\le \frac2{|1-\go^a|}
=
\frac1{|\sin(\pi a/n)|}.
\end{equation*}
Also,
\begin{equation*}
 \bigabs{\ff{\chi_I}(\q a)}
=
\lrabs{\sum_{x=0}^\ell\go^{ax}}
\le \ell+1
\le n.
\end{equation*}
For all $a\in\bbZ$ with $|a|\le n/2$ we have $|\sin(\pi a/n)|\ge2|a|/n$,
and hence
\begin{equation*}
 \bigabs{\ff{\chi_I}(\q a)}
\le \min\Bigpar{n,\frac{n}{2|a|}}
\le \frac {n}{1+|a|}.
\end{equation*}
Hence, for $|a|\le n/2$,
\begin{equation}
\label{j2}
 \bigabs{\ff{h_i}(\q a)}
=
 \bigabs{\ff{g}(\q a)}
=
 \bigabs{n\qww\ff{\chi_I}(\q a)^2}
\le \frac {1}{(1+|a|)^2}.
\end{equation}
Moreover, 
\begin{equation}
  \label{j0}
\ff{h_i}(0)=\ff{g}(0)=n\qww\ff{\chi_I}(0)^2=n\qww(\ell+1)^2\ge(\gd/2)^2.
\end{equation}
Now define $H:\bbzn^m\to\bbR$ as the tensor product of the $h_i$, so
\begin{equation}
  \label{j3}
H(\qx_1,\dots,\qx_m)\=\prodim h_i(\qx_i).
\end{equation}
By \eqref{j2}, when $|a_i|\le n/2$ for all $i$,
\begin{equation}
\label{j6}
\bigabs{\ff{H}(\q a_1,\dots,\q a_m)}
=
 \lrabs{\prodim\ff{h_i}(\q a_i)}
\le 
\prodim (1+|a_i|)\qww.
\end{equation}

Now fix $\qx_1,\ldots,\qx_m\in\bbzn$, and consider
the sum
\begin{equation}
  \label{j4}
\wXi\=\sum_{z\in \bbzn} H(z\qx_1,\dots,z\qx_m),
\end{equation}
where we abuse notation by writing $\bbzn$ for an
arbitrary subset of $\bbZ$ consisting of one element from each residue
class modulo $n$.

If this sum does not vanish, then at least one term is non-zero, which
means that $h_i(z\qx_i)\neq0$ for some $z$ and all
$i$. Fix such a $z$ and define $y_i$ by  $0\le y_i<n$
and $\q y_i=z\qx_i-\q b_i$. Then 
$g(\q y_i)=g(z\qx_i-\q b_i)=h_i(z\qx_i)\neq0$. Hence, recalling \eqref{j1},
we have $0\le y_i\le\gd n$.
Consequently (ii) holds whenever $\wXi\neq0$.

We compute $\wXi$ using Fourier inversion: 
\begin{equation}
  \label{ja1}
\begin{split}
n^m\wXi 
&=
\sum_{z\in \bbzn}\sum_{\q a_1,\dots,\q a_m\in\bbzn}
\ff H(\q a_1,\dots,\q a_m)\go^{-\sumjm zx_ja_j}
\\&
=
\sum_{\q a_1,\dots,\q a_m\in\bbzn}
\ff H(\q a_1,\dots,\q a_m)
\sum_{z\in \bbzn} \go^{-z\sumjm x_ja_j}.
\end{split}
\end{equation}
The inner sum evaluates to $n$ if $\sumjm x_ja_j\equiv0\pmod n$, and vanishes
otherwise.
Hence, 
representing $\bbzn$ by the interval
$J_n\=\set{z\in\bbZ:-n/2<z\le n/2}$,
\begin{equation}
  \label{j5}
n^{m-1}\wXi 
=
\sum_{(\q a_1,\dots,\q a_m)\in E}
\ff H(\q a_1,\dots,\q a_m),
\end{equation}
where
\begin{equation*}
  E\=\bigset{(a_1,\dots,a_m)\in J_n^m:\sumjm a_j\q x_j=0 \text{ in } \bbzn}.
\end{equation*}
Clearly, $(0,\dots,0)\in E$, and, by \eqref{j0},
\begin{equation}
  \label{j7}
\ff H(0,\dots,0)
=\prodim \ff{h_i}(0)
=\ff{g}(0)^m
\ge(\gd/2)^{2m}.
\end{equation}
If (i) does not hold, then
$(E\setminus\set{(0,\dots,0)})\cap[-C,C]^m=\emptyset$, 
and thus \eqref{j5}, \eqref{j7} and \eqref{j6} yield
\begin{equation}
  \label{j8}
\begin{split}
n^{m-1}\wXi 
&\ge 
\prodim \ff h_i(0) 
-\sum_{E\setminus\set{(0,\dots,0)}}
\lrabs{\ff H(\q a_1,\dots,\q a_m)}	
\\&
\ge
\parfrac{\gd}2^{2m} - \sum_{\bbZ^m\setminus[-C,C]^m} \prodim(1+|a_i|)\qww.
  \end{split}
\end{equation}
Since
\begin{equation}
  \label{jb3}
\sum_{\bbZ^m} \prodim(1+|a_i|)\qww=
\biggpar{\sum_{a=-\infty}^\infty(1+|a|)\qww}^m
<\infty,
\end{equation}
we can choose $C=C(m,\gd)$ such that the last sum in \eqref{j8} is less that
$(\gd/2)^{2m}$, and then \eqref{j8} implies $\wXi>0$.

In summary, choosing $C$ as above, if (i) does not hold,
then $\wXi\neq0$, which we have shown implies (ii).
\end{proof}

We can now prove our weak version of \refT{TJ1}.
\begin{theorem}\label{TJ1prime}
Theorem~\ref{TJ1} holds if the conclusion is restricted to prime values of $n$.
\end{theorem}

\begin{proof}
We may assume that $\eps\le1/2$.
Provided we choose $C\ge 2k/\eps$, as we may, if $n\le 2k/\eps$ then
\eqref{TJ1c} holds with $a_1=-a_2=n$ and the other $a_i$ equal to $0$.
We may thus assume that $n>2k/\eps$ (and so $n>8$).

Both \eqref{TJ1c} and $\kk(S,\bbzn)$ are preserved by adding a constant to all elements
of $S$, so 
we may and shall assume that $\qx_k=0$. 
We apply \refL{LJ1weak} with $m\=k-1$, $\gd\=\eps/(2k)$
and $b_i\=\ceil{in/k}$.

Suppose first that \refL{LJ1weak}(i) holds. Then $\sum_{i=1}^m a_i\qx_i=0$ for some $a_i$ with $0<\max_i|a_i|\le C$.
Set $a_k=-\sum_{i=1}^m a_i$;
then, recalling that $\qx_k=0$, we have $\sum_{i=1}^k a_i\qx_i=0$, $\sum_{i=1}^ka_i=0$ and $|a_i|\le mC$, so
our conclusion \eqref{TJ1c} holds (with a different $C$).

Suppose instead that \refL{LJ1weak}(ii) holds. Then, for some $z\in\bbZ$
we have
$z x_i\equiv x_i'\pmod n$ for $i=1,\dots,m=k-1$,
where $x_i'\=b_i+y_i \in [b_i,b_i+\gd n]$.
Now
\begin{equation}\label{j77}
\frac{in}k \le x'_i < \Bigpar{\frac ik+\gd}n+1.
\end{equation}
Furthermore, 
since $\qx_k=0$, we have
$zx_k\equiv0\pmod n$.
Let $x_0'\=0$ and $x'_k\=n$.
For $1\le i\le k$, by \eqref{j77} the gap $x'_i-x'_{i-1}$ is less
than $(1/k+\gd)n+1$.
Hence, if $S'\=\set{\qx'_i:i=1,\dots,k}\subseteq\bbzn$
and $T\=\set{0,1,\dots,\floor{(1/k+\gd)n}}$, then
$T+S'=\bbzn$.
Consequently,
\begin{equation*}
  \nn(S',\bbzn)\le|T|\le 
\Bigpar{\frac 1k+\gd}n+1,
\end{equation*}
and thus
\begin{equation*}
  \kk(S',\bbzn)
=\frac kn  \nn(S',\bbzn)
\le
1+\gd k+k/n
<1+\eps.
\end{equation*}
Since $1/k+\gd\le\tfrac12+\tfrac18<\tfrac23$, 
\eqref{j77} then implies $0<x_1'<2n/3+1<n$ and thus
$zx_i\equiv x'_1\not\equiv0\pmod n$;  hence, $z\not\equiv0\pmod
n$. 
Using the simplifying assumption that $n$ is
prime, the map $\qx\mapsto z\qx$ is an automorphism of $\bbzn$ onto itself.
Hence,
\begin{equation*}
  \kk(S,\bbzn)
=  \kk(zS,\bbzn)
=  \kk(S',\bbzn)
< 1+\eps,
\qedhere
\end{equation*}
contradicting our assumptions and completing the proof.
\end{proof}

To prove \refT{TJ1}, we need Lemma~\ref{LJ1strong} in place 
of Lemma~\ref{LJ1weak}. Recall that the only difference
is that in conclusion (ii), the former allows us to assume
that the greatest common divisor $(z,n)$ of $z$ and $n$ is equal to $1$.
Before proving Lemma~\ref{LJ1strong}, let us first show that \refT{TJ1} follows.

\begin{proof}[Proof of \refT{TJ1}.] 
  The proof is the same as the proof of \refT{TJ1prime} given above, except
  that we use \refL{LJ1strong} instead of \refL{LJ1weak}. In the case
  where \refL{LJ1strong}(ii) holds we have $(z,n)=1$,
 so $\qx\mapsto z\qx$ is still an automorphism of $\bbzn$
even though $n$ need not be prime.
\end{proof}

Let $\bbznx\=\set{x\in\bbzn:(x,n)=1}$ (again regarded as a subset of
$\bbZ$ when convenient), and let $\phi(n)\=|\bbznx|$, be the Euler phi function,
i.e., the number of integers $1\le x<n$ that are coprime to $n$.
Let $\cznx\=\chi_{\bbznx}$ denote the characteristic function of $\bbznx\subset \bbzn$. 

It remains to prove \refL{LJ1strong}; we shall deduce \refL{LJ1strong} from the following lemma, whose proof we in turn postpone.
\begin{lemma}\label{LJ4}
Given $\eta>0$, let $\cP\=\set{p \text{ prime and } p\le1+1/\eta}$ and let
$P\=\prod_{p\in\cP}p$. Then, for every $n\ge1$ and $u\in\bbZ$, either
$Pu\equiv0\pmod n$ or $|\ffznx(\q u)|<\eta\phi(n)$. 
\end{lemma}

\begin{proof}[Proof of \refL{LJ1strong}.]
Define $H$ as in the proof of \refL{LJ1weak}, but consider now
  \begin{equation*}
\xix\=\sum_{z\in\bbznx}H(z\qx_1,\dots,z\qx_m).
  \end{equation*}
As before, if $\xix\neq0$, then (ii) (in its new stronger version)
holds.

We have, in analogy with \eqref{ja1},
\begin{equation}
  \label{ja5}
\begin{split}
n^m\xix 
&=
\sum_{\q a_1,\dots,\q a_m\in\bbzn}
\ff H(\q a_1,\dots,\q a_m)
\sum_{z\in \bbznx} \go^{-z\sumjm x_ja_j}
\\
&=
\sum_{\q a_1,\dots,\q a_m\in\bbzn}
\ff H(\q a_1,\dots,\q a_m)
\ffznx \Bigpar{-\sumjm a_j\qx_j}.
\end{split}
\end{equation}
Let $\xHaa\=\prodim(1+a_i^2)\qw$, so that 
\eqref{j6} becomes
\begin{equation}\label{jb6}
|\ff H(\q a_1,\dots,\q a_m)|\le\xHaa.
\end{equation}
Note that $\sum_{\bbZ^m}\xHaa<\infty$ by \eqref{jb3},
so we may choose $C_1$ and $\eta>0$ such that
\begin{equation}\label{jb1}
\sum_{\bbZ^m\setminus[-C_1,C_1]^m}\xHaa<\frac12\parfrac{\gd}{2}^{2m}  
\end{equation}
and
\begin{equation}\label{jb2}
\eta\sum_{\bbZ^m}\xHaa<\frac12\parfrac{\gd}2^{2m}.
\end{equation}

Suppose first that
\begin{equation}
  \label{ja4}
\biggabs{\ffznx \Bigpar{-\sumjm a_j \qx_j}}>\eta\phi(n)
\end{equation}
for some $(a_1,\dots,a_m)\in[-C_1,C_1]^m\setminus\set{(0,\dots,0)}$.
Then, by \refL{LJ4}, for a certain integer $P$ depending only on
$\eta$, we have $\sumim Pa_ix_i\equiv0\pmod n$, so the conclusion (i) of \refL{LJ1strong} holds with
$a_i$ replaced by $a_i'\=Pa_i$ and $C\=PC_1$.

On the other hand, if \eqref{ja4} fails for all
$(a_1,\dots,a_m)\in[-C_1,C_1]^m\setminus\set{(0,\dots,0)}$,
then from \eqref{ja5}, \eqref{jb6}, \eqref{j7},
\eqref{jb1} and \eqref{jb2} we have
\begin{equation*}
\begin{split}
n^m\xix 
&\ge
\ff H(0,\dots,0)\phi(n)
-\sum_{\bbZ^m\setminus\set{(0,\dots,0)}}
\xHaa\biggabs{
\ffznx \Bigpar{-\sumjm a_j\qx_j}}
\\
&\ge \parfrac{\gd}{2}^{2m}\phi(n)
-\sum_{[-C_1,C_1]^m\setminus\set{(0,\dots,0)}} 
  \xHaa\eta\phi(n)
\\
&\hskip8em
-\sum_{\bbZ^m\setminus[-C_1,C_1]^m} \xHaa\phi(n)
\\
&>
\parfrac{\gd}{2}^{2m}\phi(n)\Bigpar{1-\frac12-\frac12}=0.
\end{split}
\end{equation*}
Thus $\xix\neq0$, and thus, as shown above, the conclusion
(ii) of \refL{LJ1strong} holds.
\end{proof}

It remains to prove our estimate \refL{LJ4} of the Fourier transform
$\ffznx$ of $\cznx=\chi_{\bbznx}$. First, we state (and
this time immediately prove) a final lemma.

\begin{lemma}
  \label{LJ3}
  Let $n=\prodir \pie$ be the prime factorization of $n\ge 2$. Then
  for any $u\in\bbZ$ we have $\ffznx(\q u)=\prodir\psi_{\pie}(u)$, where
  \begin{equation*}
	\psi_{p^e}(u)\=
	\begin{cases}
0, & p^{e-1}\ndela u;
\\
-p^{e-1},& p^{e-1}\dela u,\;p^e\ndela u;
\\
p^e-p^{e-1},&p^e\dela u.	  
	\end{cases}
  \end{equation*}
\end{lemma}

\begin{proof}
Set $f(y_1,\dots,y_r)\=\sumir y_in/\pie$, and define 
$\q f:\prodir\bbZ_{\pie} \to\bbzn$ by $\q f(\q y_1,\dots,\q
y_r)\=\qa{f(y_1,\dots,y_r)}$.
Note that $\q f$ is well defined. Furthermore,
\begin{equation}\label{j9}
f(y_1,\dots,y_r)\equiv y_in/\pie \mod\pie,\qquad i=1,\dots,r,  
\end{equation}
which
implies that $\q f$ is injective and thus a bijection. (In fact,
$\q f$ is the map whose existence is guaranteed by the Chinese Remainder Theorem.)
Moreover, by \eqref{j9}, $p_i\dela f(y_1,\dots,y_r)\iff p_i\dela y_i$,
and thus 
$(f(y_1,\dots,y_r),n)=1\iff(y_i,\pie)=1$ for every $i$; consequently,
$\q f$ is also a bijection $\prodir\bbZ_{\pie}^*\to\bbznx$.
Hence,
\begin{align*}
  \ffznx(\q u)
&=
\sum_{\q x\in\bbznx} \ex{2\pi\ii ux/n}
\\
&=
\sum_{\q y_i\in\bbZ_{\pie}^*,\,i=1,\dots,r} 
\ex{2\pi \ii uf(y_1,\dots,y_r)/n}
\\
&=
\sum_{\q y_i\in\bbZ_{\pie}^*,\,i=1,\dots,r} 
\prodir \ex{2\pi \ii u y_i/\pie}
\\
&=
\prodir 
\sum_{\q y\in\bbZ_{\pie}^*} 
\ex{2\pi \ii u y/\pie}
.
\end{align*}
For any prime $p$ and any $e\ge1$, $\bbZ_{p^e}\xxx$ can be represented
  by $\set{y:1\le y\le p^e}\setminus\set{pz:1\le z\le p^{e-1}}$, so
  \begin{equation*}
\sum_{\q y\in\bbZ_{p^e}\xxx}\ex{2\pi\ii uy/p^e}
=
\sum_{y=1}^{p^e} \ex{2\pi\ii u y/p^e}	
-
\sum_{z=1}^{p^{e-1}} \ex{2\pi\ii u z/p^{e-1}}	
=\psi_{p^e}(u),
  \end{equation*}
since
\begin{equation*}
\sum_{y=1}^{p^k} \ex{2\pi\ii u y/p^k}	
=
\begin{cases}
  p^k, & p^k\dela u,
\\
  0, & p^k\ndela u,
\end{cases}
\end{equation*}
for any $k\ge 0$.
\end{proof}

\refL{LJ4} now follows easily.
\begin{proof}[Proof of \refL{LJ4}.]
  Using the notation of \refL{LJ3}, we see that if $\pie\ndela u$ for
  some prime factor $p_i\notin\cP$ of $n$, then,
using $\phi(n)=\ffznx(0)=\prodir(\pie-p_i^{e_i-1})$, we have
  \begin{equation}
	\label{ja2}
\frac{\lrabs{\ffznx(\q u)}}{\phi(n)}
\le
\frac{|\psi_{\pie}(u)|}{\pie-p_i^{e_i-1}}
\le\frac1{p_i-1}<\eta.
  \end{equation}
Assume $\lrabs{\ffznx(\q u)}\ge\eta\phi(n)$. Then, by \eqref{ja2},
$\pie\mid u$ for every $p_i\notin\cP$. Furthermore, for every $p_i\in\cP$ we have
$p_i^{e_i-1}\dela u$, since otherwise $\ffznx(\q u)=0$ by
\refL{LJ3}. Hence $n\dela Pu$. 
\end{proof}

As noted above, \refL{LJ4} was the only missing piece of the puzzle; the proof of~\refL{LJ1strong}
is now complete, and hence that of \refT{TJ1}. As shown in \refS{SZsmall},
\refT{TJ1Z} and Corollaries~\ref{CJ1} and~\ref{CJ2} follow.

\newcommand\AAP{\emph{Adv. Appl. Probab.} }
\newcommand\JAP{\emph{J. Appl. Probab.} }
\newcommand\JAMS{\emph{J. \AMS} }
\newcommand\MAMS{\emph{Memoirs \AMS} }
\newcommand\PAMS{\emph{Proc. \AMS} }
\newcommand\TAMS{\emph{Trans. \AMS} }
\newcommand\AnnMS{\emph{Ann. Math. Statist.} }
\newcommand\AnnPr{\emph{Ann. Probab.} }
\newcommand\CPC{\emph{Combin. Probab. Comput.} }
\newcommand\JMAA{\emph{J. Math. Anal. Appl.} }
\newcommand\RSA{\emph{Random Struct. Alg.} }
\newcommand\ZW{\emph{Z. Wahrsch. Verw. Gebiete} }
\newcommand\DMTCS{\jour{Discr. Math. Theor. Comput. Sci.} }

\newcommand\AMS{Amer. Math. Soc.}
\newcommand\Springer{Springer-Verlag}
\newcommand\Wiley{Wiley}

\newcommand\vol{\textbf}
\newcommand\jour{\emph}
\newcommand\book{\emph}
\newcommand\inbook{\emph}
\def\no#1#2,{\unskip#2, no. #1,} 
\newcommand\toappear{\unskip, to appear}

\newcommand\webcite[1]{\hfil
   \penalty0\texttt{\def~{{\tiny$\sim$}}#1}\hfill\hfill}
\newcommand\webcitesvante{\webcite{http://www.math.uu.se/~svante/papers/}}
\newcommand\arxiv[1]{\webcite{http://arxiv.org/abs/#1}}

\def\nobibitem#1\par{}

\end{document}